\documentclass[11pt]{article}  
\usepackage{tikz}

\usetikzlibrary{shapes,arrows}
\usepackage{caption}
\usepackage{graphicx}

\usepackage{latexsym}
\usepackage[english]{babel}
\usepackage{amsmath}
\usepackage{amsfonts}
\usepackage{amssymb}
\usepackage{makeidx}
\usepackage[utf8]{inputenc}
\usepackage{lmodern}
\usepackage{verbatim}
\usepackage[pagewise]{lineno}
\usepackage{amsthm}
\usepackage{tikz-cd} 

\usepackage[square,numbers]{natbib}
\newtheorem{theorem}{Theorem}[section]
\newtheorem{corollary}[theorem]{Corollary}

\newtheorem{lemma}[theorem]{Lemma}
\newtheorem{proposition}[theorem]{Proposition}

\theoremstyle{definition}
\newtheorem{definition}[theorem]{Definition}
\newtheorem{remark}[theorem]{Remark}
\newtheorem{example}[theorem]{Example}

\newcommand{\R}{\ensuremath{\mathbb{R}}}

\newcommand{\C}{\ensuremath{\mathcal{C}}}
\newcommand{\D}{\ensuremath{\mathcal{D}}}
\newcommand{\Es}{\ensuremath{\mathbb{S}}}

\newcommand{\Le}{\ensuremath{\mathcal{L}}}

\numberwithin{equation}{section}

\DeclareMathOperator{\arcsinh}{arcsinh}

\begin{document}
	\title{Jacobi fields in nonholonomic mechanics}
	\author{
		{\bf\large Alexandre Anahory Simoes}\hspace{2mm}
		\vspace{1mm}\\
		{\small  Instituto de Ciencias Matem\'aticas (CSIC-UAM-UC3M-UCM)} \\
		{\small C/Nicol\'as Cabrera 13-15, 28049 Madrid, Spain}\\
		{\it\small e-mail: \texttt{alexandre.anahory@icmat.es }}\\
		\vspace{2mm}\\
		{\bf\large Juan Carlos Marrero}\hspace{2mm}
		\vspace{1mm}\\
		{\it\small 	ULL-CSIC Geometr{\'\i}a Diferencial y Mec\'anica Geom\'etrica,}\\
		{\it\small  {Departamento de Matem\'aticas, Estad{\'\i}stica e I O, }}\\
		{\it\small  {Secci\'on de Matem\'aticas, Facultad de Ciencias}}\\
		{\it\small  Universidad de la Laguna, La Laguna, Tenerife, Canary Islands, Spain}\\
		{\it\small e-mail: \texttt{{jcmarrer@ull.edu.es}} }\\
		\vspace{2mm}\\
		{\bf\large David Martín de Diego}\hspace{2mm}
		\vspace{1mm}\\
		{\small  Instituto de Ciencias Matem\'aticas (CSIC-UAM-UC3M-UCM)} \\
		{\small C/Nicol\'as Cabrera 13-15, 28049 Madrid, Spain}\\
		{\it\small e-mail:  \texttt{david.martin@icmat.es} }\\		
	}
	
	\date{}

	\maketitle
	
	\vspace{0.5cm}
	\begin{abstract}
		In this paper, we define Jacobi fields for nonholonomic mechanics using a similar characterization than in Riemannian geometry. We give explicit conditions to find Jacobi fields and finally we find the nonholonomic Jacobi equations in two equivalent ways:  the first one, using an appropriate complete lift of the nonholonomic system and, in the second one,  using the curvature and torsion of the associated nonholonomic connection.

	\end{abstract}
	
	\let\thefootnote\relax\footnote{\noindent AMS {\it Mathematics Subject Classification ({2020})}. Primary 70G45; Secondary  53B20, 53C21,
		37J60, 70F25
	}
	\tableofcontents
	
	\section{Introduction}
	
	Jacobi fields play a fundamental role in Riemannian geometry since they  provide a means of describing how fast the geodesics starting from a given point spread apart. This behaviour is controlled by the curvature of the Riemannian metric \citep*{docarmo,ON83,LeeR}. 
	More concretely, 
	let $(Q, g)$ be a Riemannian manifold where $Q$ is a differentiable manifold and $g$ a   given Riemannian metric, then a 
	Jacobi field  is  vector field $W: I\rightarrow TQ$ along a geodesic $c: I\rightarrow Q$ satisfying the  equation (Jacobi equation)
	\begin{equation}\label{jacobi-g}
	\frac{D^2 W}{dt^2}+R(W(t), \dot{c}(t))\dot{c}(t)=0
	\end{equation}
	where $R$ is the curvature of the Riemannian manifold.
	This equation measures the effect of curvature on  one-parameter families of geodesics being  Jacobi fields their infinitesimal generators. In this last sense, we can interpret Jacobi equation (\ref{jacobi-g}) as a second order  differential equation satisfied by the variation fields of one-parameters families of geodesics. 
	Another interesting approach is to derive an appropriate  variational principle which gives simultaneously the geodesic and Jacobi  equations for the Riemannian manifold, that is, to see the Jacobi equation as the Euler-Lagrange equations for a suitable Lagrangian  function defined now on $TTQ$ {(see \citep*{YaIs73})}.  
	
	
	The study of Jacobi fields is an important step to analyse the local and global geometry and topology of a Riemannian manifold. For instance, such important issues as the existence of conjugate points, the minimizing character of geodesics,  singularities of the exponential map where it fails to be a local diffeomorphism...
{(see, for instance, \citep*{docarmo,ON83,LeeR})}.
	
	Summarizing, the main steps to characterize Jacobi fields in Riemannian geometry are (see Table \ref{tab} for more details): 
	\begin{enumerate}
		\item Characterization of Jacobi fields in terms of infinitesimal geodesic variations.
		\item To show that  Killing vector fields $X\in {\mathfrak X}(Q)$ are  Jacobi fields along any geodesic $c$ in $Q$.
		\item Characterization of Jacobi fields as trajectories of the kinematic Lagrangian system obtained by lifting the metric to the tangent bundle.
		\item Derivation of the Jacobi equation in terms of the curvature of the Riemannian metric. 
		\item Study of conjugate points and relation with minimizing properties of geodesics.
	\end{enumerate}
	Given the importance of Jacobi fields, there has been great interest in generalizing these results to different situations, for example, to general second order differential equations (SODE's) in \citep*{CM92} {using the dynamical covariant derivative and the Jacobi endomorphism associated with the SODE \citep*{MaCaSa} (see also \citep*{HaMe} and the references therein}), to semi-Riemannian geometry \citep*{ON83}, to sub-Riemannian and Finsler geometry \citep*{agrachev, Bari}, to the Lie algebroid setting \citep*{CGM2015}, to skew-symmetric algebroids \citep*{Joz}, etc.

	However, the case of systems subjected to nonholonomic constraints has not be properly considered in the previous literature. Roughly speaking a nonholonomic system consists of a Lagrangian function $L: TQ\rightarrow \R$  and a nonintegrable distribution $\D$ where the dynamics is governed by the Lagrange-d'Alembert principle which is non-variational in the standard terminology {\citep*{Bloch, Cortes, Neimark, Pacific} (see also \citep*{Ma} for a discussion on the validity of the Lagrange-d'Alembert principle and \cite*{GG2008} for a general discussion on variational calculus with constraints). Usually, the Lagrangian is of the form $L\equiv T$ or $L\equiv T-V$ where $T$ is the kinetic energy associated to a Riemannian metric and $V$ is a potential function. Nonholonomic systems are present in many areas of applied research like wheeled vehicles, robotics, satellite dynamics and its dynamic is very intriguing from a qualitative point of view.  These are some of the reasons for the vast literature about this topic${}^1$\footnote{1. As a simple checking, we have found almost 2000 references in \textit{MathScinet} with the words nonholonomic or non-holonomic in the title}.
	
	One of the most successful approaches to  understanding the dynamics of nonholonomic  systems has been the use of differential geometry and, in particular, of Riemannian  geometry (see, for instance, {\citep*{Synge28, koiller, Lewis98, CaRa, cantrijn-e, Cortes, GrLeMaMa, BLMMM2011, GaMa}}). 
	One of the main ingredients is to introduce a nonholonomic connection obtained using the orthogonal projector associated to the decomposition $TQ=\D\oplus \D^\perp$ where $\D^\perp$ is the orthogonal distribution for the Riemannian metric $g$ {(see \citep*{Lewis98})}. 
	
	With the present paper, we start the geometric program of introducing some important concepts of standard Riemannian geometry to the nonholonomic setting. In particular, in this paper, we will introduce the notion of Jacobi fields and the corresponding Jacobi equation. More concretely, following the same lines as in the unconstrained case (see Table \ref{tab} for more details):
	\begin{enumerate}
		\item We have defined nonholonomic  Jacobi fields in terms of infinitesimal nonholonomic geodesic variations in Definition \ref{definition31}. We remark that, following our natural definition, a nonholonomic Jacobi field along a nonholonomic trajectory $c$ is not, in general, a section of the constraint distribution $D$ along $c$. This is an important difference with previous approaches to the notion of a nonholonomic Jacobi field.
		\item We have given in Theorem  \ref{jacobikilling} new   results to  explicitly find  Jacobi fields.  In particular,  {every Killing vector field $X\in {\mathfrak X}(Q)$ which is an infinitesimal symmetry of $D$ is a non-holonomic Jacobi field along any nonholonomic geodesic $c$ (see Corollary \ref{killingcorollary})}.
		\item We have characterized  Jacobi fields as trajectories of a  lifted nonholonomic system in Theorem \ref{MainTheorem}.
		\item Finally, we have derived  the nonholonomic Jacobi equation in terms of the curvature and torsion of the corresponding nonholonomic connection in Theorem \ref{Jacobi:equation}. 
	\end{enumerate}	
	To do all this, we widely use the theory of complete and vertical lifts to the tangent bundle (see \cite*{YaIs73}; see also \cite*{GrUr} for an extension of the theory to general algebroids).
	
	On the other hand, to preserve as much as possible the Riemannian geometric flavour we start our study with nonholonomic systems of kinematic type, but in Appendix \ref{appd} we extend the previous results to the case of a nonholonomic system where the Lagrangian is of the form $L=T-V$. 
	
		{\small
		\begin{table}[tbhp]
			\centering
			\begin{tabular}{|l|l|}
				\hline%
				{\bf \large Riemannian geometry}&
				{\bf \large Kinematic nonholonomic mechanics} \\
				\hline \hline
				&\\
				\parbox[t]{.47\linewidth}{A vector field $W:I\rightarrow TQ$ along a geodesic  $c:I\rightarrow Q$ is said to be a \textit{ Jacobi field} for the Riemannian manifold $(Q, g)$ if it is the infinitesimal variation vector field of a family of geodesics} & \parbox[t]{.47\linewidth}{ A vector field $W:I\rightarrow TQ$  along a nonholonomic trajectory  $c:I\rightarrow Q$ is said to be a \textit{nonholonomic Jacobi field} for the system $(L_g,\D)$ if it is the infinitesimal variation vector field of a family of nonholonomic trajectories  (see {\bf Definition \ref{definition31}})}\\
				\hline
				&
				\\
				\parbox[t]{.47\linewidth}{Every Killing vector field $W$ for the Riemannian metric $g$ is a Jacobi field along any geodesic}
				&
				\parbox[t]{.47\linewidth}{ Every Killing vector field $W$ for the Riemannian metric $g$ which is an  infinitesimal symmetry of $\D$  is a nonholonomic Jacobi field for any nonholonomic solution\\ (see {\bf Corollary \ref{killingcorollary}})}
				\\
				\hline
				&
				\\			\parbox[t]{.47\linewidth}{The trajectories of the Lagrangian system $L_{g^c}: TTQ\rightarrow \R$ are just the Jacobi fields for the Riemannian manifold $(Q, g)$}
				&
				\parbox[t]{.47\linewidth}{The trajectories of the nonholonomic system 
					$(L_{g^{c}},\D^{c})$ are just the Jacobi fields for the nonholonomic system determined by $(L_{g},\D)$\\ (see {\bf Theorem \ref{MainTheorem})}}\\
				\\
				\hline
				&
				\\
				\parbox[t]{.47\linewidth}{	$W$ is a Jacobi field if and only if 
					\[
					\frac{D^2 W}{dt^2}+R(W(t), \dot{c}(t))\dot{c}(t)=0
					\]
					or, equivalently, 
					\[\nabla^{g}_{\dot{c}}\nabla^{g}_{\dot{c}}W+R(W,\dot{c})\dot{c}=0
					\]}
				& \parbox[t]{.47\linewidth}{$W$ is a nonholonomic Jacobi field if and only if 
					\begin{eqnarray*}
						&&\nabla^{nh}_{\dot{c}}\nabla^{nh}_{\dot{c}}W+\nabla^{nh}_{\dot{c}}T^{nh}(W,\dot{c})\\
						&&+R^{nh}(W,\dot{c})\dot{c}=0, \quad \dot{W}(t)\in \D^{c}.
					\end{eqnarray*}
					(see {\bf Theorem  \ref{Jacobi:equation}})
				}
				\\
				\hline
				
			\end{tabular}
			\medskip
			\caption{Comparative between Jacobi fields for Riemannian geometry and kinematic nonholonomic systems }
			\label{tab}
		\end{table}
	}

	{The paper is organized as follows. In Section \ref{review-non-holonomic-systems}, we will review the main results on the geometric formulation of nonholonomic systems subjected to linear constraints (which will be used in the rest of the paper). We also extend some basic results on nonholonomic systems of kinetic type to the more general case when the kinetic energy is not induced by a Riemannian metric but a pseudo-Riemanninan metric. In Section \ref{non-holonomic-Jacobi-fields}, we introduce the notion of a nonholonomic Jacobi field for a kinematic nonholonomic system. Then, we prove that such systems may be lifted to kinematic nonholonomic systems (on the double tangent bundle) with kinetic energy induced by a pseudo-Riemanninan metric and that the nonholonomic trajectories of the lifted system are just the Jacobi fields of the original system. We also prove that the Jacobi fields satisfy the nonholonomic Jacobi equation which is given in terms of the covariant derivative, the curvature and the torsion associated with the nonholonomic connection. Along the paper, we have included several examples to illustrate the potential of our results.} 
	
	In addition, to make the paper self-contained and easily readable we have included four appendices also containing original results. The Appendix \ref{appa} introduces some well-know results about  complete and vertical lifts of vector fields and 1-forms. In  Appendix \ref{appb} we review some results about complete and vertical lifts in Riemannian and pseudo-Riemanninan geometry relating these lifted objects with the Poincar\'e-Cartan forms. In Appendix \ref{appc} we discuss the properties of the complete lift of a Lagrangian system of kinetic type. Finally in Appendix \ref{appd}, as we said before, we characterize Jacobi fields for nonholonomic mechanical systems where the Lagrangian is of the form $L=T-V$.

\section{Nonholonomic systems subjected to linear constraints}\label{review-non-holonomic-systems}

A \textit{Lagrangian mechanical system} is a pair formed by a smooth manifold $Q$ called the \textit{configuration space} and a smooth function $L:TQ\rightarrow\R$ on its tangent bundle called the \textit{Lagrangian} \citep*{AM78}. If the system is not subjected to any constraint or external forces, a \textit{motion} of the mechanical system is a solution of the \textit{Euler-Lagrange equations}, whose expression on natural coordinates relative to any chart $(q^{i})$ for $Q$ is 
\begin{equation} \label{EL}
\frac{d}{dt}\left(\frac{\partial L}{\partial \dot{q}^{i}}\right) - \frac{\partial L}{\partial q^{i}}=0, \quad 1\leqslant i \leqslant \text{dim}(Q). 
\end{equation}

We can also express these equations in a more geometric language \citep*{LR89}. Let $\tau_Q:TQ\rightarrow Q$ be the canonical tangent projection, $\Delta$ be the \textit{Liouville vector field} on $TQ$ defined by
\begin{equation*}
\Delta(u)=\left.\frac{d}{dt}\right|_{t=0}(u + t u)=(u)^{v}_{u},
\end{equation*}
and $S:TTQ\rightarrow TTQ$ be the \textit{vertical endomorphism} defined by
\begin{equation*}
	S(X_{u})=((T_{u} \tau_{Q}) (X_{u}))^{v}_{u},
\end{equation*}
where $(\cdot)^{v}_{u}:T_{q}Q\rightarrow T_{u}(T_{q}Q)$ denotes the \textit{vertical lift} at $u\in T_{q} Q$. These maps have a much more immediate interpretation when are written in coordinates where $$\Delta(v^{i}\frac{\partial}{\partial q^{i}})=v^{i}\frac{\partial}{\partial \dot{q}^{i}} \  \hbox{   and   }  \ S(X^{i}\frac{\partial}{\partial q^{i}}+X^{n+i}\frac{\partial}{\partial \dot{q}^{i}})=X^{i}\frac{\partial}{\partial \dot{q}^{i}}\; .$$

When the function $L$ is \textit{regular}, that is, the matrix $\text{Hess}(L):=(\frac{\partial^2 L}{\partial \dot{q}^i \partial \dot{q}^j})$ is non-singular, equations \eqref{EL} may be written as a system of second-order differential equations obtained by computing the integral curves of the unique vector field $\Gamma_L$ satisfying
\begin{equation}
i_{\Gamma_L}\omega_L=dE_L, \label{Lvf}
\end{equation}
where $\omega_L =-d(S^* dL)$ and {$E_L =\Delta (L)-L$} are the \textit{Poincaré-Cartan two-form} and the \textit{Lagrangian energy}, respectively. The one-form $\theta_{L}=S^* dL$ is the \textit{Poincaré-Cartan one-form}. The vector field $\Gamma_L$ is a SODE vector field on $Q$ and it is called the \textit{Lagrangian vector field}. Regularity of $L$ is equivalent to $\omega_L$ being symplectic and therefore to uniqueness of solution for equation \eqref{Lvf}. In effect, its local expression is
$$\omega_L=\frac{\partial^2 L}{\partial \dot{q}^i \partial q^j} d q^{i}\wedge d q^{j}+\frac{\partial^2 L}{\partial \dot{q}^i \partial \dot{q}^j} dq^{i}\wedge d\dot{q}^{j},$$
from where we deduce that $\omega_L$ is symplectic if and only if $\text{Hess}(L)$ is non-singular (for more details see \citep*{AM78,LR89}).

\subsection{Nonholonomic mechanics}

A \textit{nonholonomic constraint} on a regular mechanical system $(Q,L)$ is a nonintegrable distribution $\D$ on $Q$ and a \textit{nonholonomic mechanical system} is a triple $(Q,L,\D)$ such that the system is subjected to the constraint in the sense that the velocity vectors of motions belong to the distribution $\D$ \citep*{Bloch}.

Note that if the distribution was integrable, then the manifold $Q$ would be foliated by immersed submanifolds of $Q$ whose tangent space at each point coincides with the subspace given by the distribution at that point. Hence, motions of these systems would be confined to a submanifold $N\subseteq Q$. In this way, we obtain a mechanical system without constraints formed by $(N,L|_{N})$. This class of constraints is called \textit{holonomic constraints}.

Locally, the nonholonomic constraints are given by a set of $n-k$ equations that are linear on the velocities $$\mu^{a}_{i}(q)\dot{q}^{i}=0,$$ where $1\leqslant a\leqslant n-k$, where $k$ is the rank of the distribution $\D$ and $n$ is the dimension of the manifold $Q$. Geometrically, these equations define the vector subbundle $\D^{o}\subseteq T^{*}Q$, called the \textit{annihilator} of $\D$, spanned at each point by the one forms $\{\mu^{a}\}$ locally given by $\mu^{a}=\mu^{a}_{i}(q) dq^{i}$.

A motion of the nonholonomic mechanical system is a solution of the \textit{Lagrange-d'Alembert equations}, whose local expression is
\begin{equation} \label{LdA}
	\left\{\begin{array}{l}
		\displaystyle \frac{d}{dt}\left(\frac{\partial L}{\partial \dot{q}^{i}}\right) - \frac{\partial L}{\partial q^{i}}=\lambda_{a}\mu^{a}_{i}(q) \\ \mu^{a}_{i}(q)\dot{q}^{i}=0,
	\end{array}\right.
\end{equation}
for some Lagrange multipliers $\lambda_{a}$, which may be determined with the help of the constraint equations (see \citep*{VF72,LM94, LMdD1996, CeMaRa2001, CdLMdDM2003} for a first geometrical approach to nonholonomic mechanics).

Nonholonomic mechanics can also be described in a more geometric fashion. Consider the geometric equations
{\begin{equation}\label{nhLvf}
	\left\{\begin{array}{l}
		i_{\Gamma_{(L, D)}}\omega_L-dE_L\in \Gamma(F^o) \\
		\Gamma_{(L, \D)} \in {\frak X}(\D),
	\end{array}\right.
\end{equation}}
{where $\Gamma(F^0)$ is the space of sections of $F^0$, with} $F^{o}=S^*((T\D)^{o})$ the annihilator of a distribution $F$ on $TQ$ defined along $\D$. If the nonholonomic system is \textit{regular}, that is, the following conditions are satisfied (again see \citep*{LMdD1996}):
\begin{enumerate}
\item $\text{dim}(T_{v}\D)^{o}=\text{dim}(F_{v}^{o})$ (\textit{admissibility condition});
\item $T_{v}\D\cap (\sharp_{\omega_{L}})_{v}(F_{v}^{o})=\{0\}$ for all $v\in\D$ (\textit{compatibility condition}),
\end{enumerate}
then equations \eqref{nhLvf} have a SODE $\Gamma_{(L, \D)}$ as a unique solution on $\D$ and its integral curves satisfy equations \eqref{LdA}. Here, $\sharp_{\omega_{L}}:T^{*}(TQ)\rightarrow T(TQ)$ is the \textit{sharp isomorphism} and it is the inverse map of the \textit{flat isomorphism} defined by $\flat_{\omega_{L}}(X)=i_{X}\omega_{L}$.

The following theorem is a useful sufficient condition to prove that the nonholonomic system is regular (see \citep*{LMdD1996}).
\begin{theorem}
	If the Lagrangian $L$ has either a positive definite or a negative definite Hessian matrix $\text{Hess}(L)$, then the nonholonomic system is also regular.
\end{theorem}

To each of the one-forms $\mu^{a}$ associate the fiberwise linear function $\widehat{\mu}^{a}:TQ\rightarrow \R$ defined by $\widehat{\mu}^{a}(v_q)=\langle \mu^{a}(q),v_q \rangle$, for $v_q\in T_q Q$.
In local coordinates, equation \eqref{nhLvf} may be written like
\begin{equation*}
	i_{\Gamma_{(L, \D)}}\omega_L-dE_L=\lambda_{a} S^{*}(d\widehat{\mu}^{a})=\lambda_{a} \mu^{a}_{i} dq^{i},
\end{equation*}
for some Lagrange multipliers $\lambda_{a}$. Therefore, a solution $\Gamma_{(L, \D)}$ of \eqref{nhLvf} is of the form $\Gamma_{(L, \D)}=\Gamma_{L}+\lambda_{a} Z^{a}$, where $Z^{a}=\sharp_{\omega_{L}}(\mu^{a}_{i} dq^{i})$. The Lagrange multipliers may be computed by imposing the tangency condition in \eqref{nhLvf}, which is equivalent to
\begin{equation*}
	0=\Gamma_{(L, \D)}(\widehat{\mu}^{b})=\Gamma_{L}(\widehat{\mu}^{b})+\lambda_{a} Z^{a}(\widehat{\mu}^{b}), \ \ \text{for} \ b=1,...,n-k.
\end{equation*}
This equation has a unique solution for the Lagrange' multipliers if and only if the matrix $\mathcal{C}^{a b}=Z^{a}(\widehat{\mu}^{b})$ is invertible at all points of $\D$. In fact, in \citep*{LMdD1996}, the authors prove the last condition is equivalent to the compatibility condition. In local coordinates, one finds that
\begin{equation}\label{matrixC}
	C^{a b}=\mu_{i}^{a}W^{i j}\mu_{j}^{b},
\end{equation}
where $W^{i j}$ is the inverse matrix of $W_{i j}=\text{Hess}(L)$.

Recall from symplectic geometry that $F^{\bot}=\sharp_{\omega_{L}}(F^{o})$ for any distribution $F$, where $\bot$ denotes the symplectic orthogonal relative to $\omega_{L}$. Hence, the compatibility condition also implies the Whitney sum decomposition $$T(TQ)|_{\D}=T\D\oplus F^{\bot},$$
to which we may associate two complementary projectors $\bar{P}:T(TQ)|_{\D}\rightarrow T\D$ and $\bar{P}':T(TQ)|_{\D}\rightarrow F^{\bot}$ with coordinate expressions
\begin{equation*}
	\bar{P}(X)=X-\mathcal{C}_{a b}\widehat{\mu}^{b}(X)Z^{a}, \quad \bar{P}'(X)=\mathcal{C}_{a b}\widehat{\mu}^{b}(X)Z^{a}.
\end{equation*}

\begin{proposition}{\citep*{LMdD1996}}
	The nonholonomic dynamics is given by $$\Gamma_{(L, \D)}=\bar{P}( \left.\Gamma_{L}\right|_{\D} ).$$
\end{proposition}

Indeed, under all the assumptions we have considered so far, we can compute the Lagrange multipliers to be $$\lambda_{a}=-\mathcal{C}_{a b}\Gamma_{L}(\Phi^{b}),$$ from where the result follows.

There are several typical examples of nonholonomic systems. Some of them are the nonholonomic particle, the vertical rolling disk, the knife edge, the Chaplygin sleigh or the rolling ball on the table (cf. \citep*{Bloch}).

\begin{example}\label{Nhparticle}
	We will introduce here an example of a simple nonholonomic system to which we will get back all along the text: the \textit{nonholonomic particle}. Consider a mechanical system in the configuration manifold $Q=\R^3$ defined by the Lagrangian
	$$L_{g}(x,y,z,\dot{x},\dot{y},\dot{z})=\frac{1}{2}(\dot{x}^2 +\dot{y}^2 +\dot{z}^2)$$
	and subjected to the nonholonomic constraint $\dot{z}-y \dot{x}=0$. The one-form $\mu=dz-y \ dx$ spans the vector subbundle $\D^{o}$, which is the annihilator of the distribution $$\D=\text{span}\left\{\frac{\partial}{\partial x}+y \frac{\partial}{\partial z}, \frac{\partial}{\partial y} \right\}.$$
	Then the equations of motion of this system are given by Lagrange-d'Alembert equations \eqref{LdA}, which in this case hold
	\begin{equation}
		\begin{cases}
			\ddot{x}=- \lambda y \\
			\ddot{y}=0 \\
			\ddot{z}=\lambda \\
			\dot{z}-y\dot{x}=0
		\end{cases}
		\quad
		\Rightarrow
		\quad
		\begin{cases}
			\ddot{x}=-y\frac{\dot{x}\dot{y}}{1+y^2} \\
			\ddot{y}=0 \\
			\ddot{z}=\frac{\dot{x}\dot{y}}{1+y^2} \\
			\dot{z}-y\dot{x}=0,
		\end{cases}
	\end{equation}
	where the value of $\lambda$ is computed with the help of the constraints. This is in fact possible because the nonholonomic system is regular. This can be immediately seen from the fact that $\text{Hess}(L)$ is positive definite.
\end{example}

\subsection{{Kinematic nonholonomic systems with kinetic energy induced by a pseudo-Riemanninan metric}}

In \citep*{Lewis98}, the author expresses the dynamics of a purely kinetic Lagrangian system with nonholonomic constraints as the geodesic equation of a non-Levi-Civita connection on $Q$ {(see also \citep*{Synge28})}. This connection, which we will denote by $\nabla^{nh}$, will help us define a \textit{nonholonomic Jacobi equation}. It is defined as
\begin{equation}\label{nhconnection}
\nabla^{nh}_{X} Y:=P(\nabla_{X}^{g} Y)+\nabla^{g}_{X}[P'(Y)],
\end{equation}
where $g$ is a Riemannian metric, $\nabla^{g}$ its Levi-Civita connection, $P:TQ\rightarrow \D$ is the associated  orthogonal projector onto the distribution $\D$ and $P':TQ\rightarrow\D^{\bot}$ is the orthogonal projector onto $\D^{\bot}$, the orthogonal distribution.

This connection is not symmetric in general neither it is compatible with the metric. Nevertheless, it satisfies the more restricted condition of compatibility with the metric over sections of $\D$ (see \citep*{Lewis98}), i.e.,
\begin{equation}\label{Dcompatibility}
	X(g(Y,Z))=g(\nabla^{nh}_{X} Y,Z)+g(Y,\nabla^{nh}_{X} Z), \quad \forall X,Y,Z\in\Gamma(\D).
\end{equation}

An other important property that we will use is that if $Y\in\Gamma(\D)$, then $\nabla^{nh}_{X} Y=P(\nabla_{X}^{g} Y)$ for any vector field $X\in\mathfrak{X}(Q)$.

Let $h$ be a non-degenerate symmetric $(0,2)$-tensor on $Q$, that is, $h$ is a pseudo-Riemannian metric on $Q$. Then the Levi-Civita connection of $h$ is well-defined and, moreover, in the presence of a distribution $\D$ on $Q$ the orthogonal projectors $P:TQ\rightarrow \D$ and $P':TQ\rightarrow\D^{\bot}$ are well-defined if and only if $\D\cap\D^{\bot}=\{0\}$. In the sequel, we will need  to extend the definition of the connection \eqref{nhconnection} to these cases.

If $h$ is a pseudo-Riemannian metric on $Q$, we will denote by $L_{h}:TQ\rightarrow\R$ the Lagrangian function associated with $h$ defined by
\begin{equation*}
	L_{h}(v)=\frac{1}{2}h(v,v), \quad v\in TQ.
\end{equation*}

Note that, using the fact that $h$ is non-degenerate, we deduce that the Poincaré-Cartan 2-form $\omega_{L_{h}}=-d\theta_{L_{h}}$ is symplectic (see Appendix B).

We will show first a useful property, relating the symplectic and the metric structures, that we will need to prove the next theorem.
\begin{lemma}\label{omegametriclemma}
	Let $h$ be a pseudo-Riemannian metric on $Q$, $L_{h}$ the associated kinetic Lagrangian and $\omega_{L_{h}}$ the associated symplectic form on $TQ$. Denote by $\sharp_{h}:T^{*}Q\rightarrow TQ$ and $\sharp_{\omega_{L_{h}}}:T^{*}TQ\rightarrow TTQ$ the musical isomorphisms with respect to the metric and symplectic form, respectively. Then for any $\alpha\in\Omega^{1}(Q)$ we have
	\begin{equation*}
	\sharp_{\omega_{L_{h}}}\circ \alpha^{v}=-(\sharp_{h}\circ \alpha)^{v},
	\end{equation*}
	where $\alpha^{v}\in\Omega^{1}(TQ)$ and $(\sharp_{h}\circ \alpha)^{v}\in \mathfrak{X}(TQ)$ are the vertical lifts to $TQ$ of the 1-form $\alpha$ and the vector field $\sharp_{h}\circ \alpha$, respectively (see Appendix A).
\end{lemma}

\begin{proof}
	It follows using the first relation in \eqref{omegalift} (see Appendix B).
\end{proof}

Now we will see that a distribution is non-degenerate with respect to a pseudo-Riemannian metric if and only if the induced nonholonomic system is regular.

\begin{theorem}\label{regularityth}
	Given a pseudo-Riemannian metric $h$ on a manifold $Q$ and a distribution $\D$ the following are equivalent:
	\begin{enumerate}
		\item $\D \cap \D^{\bot}=\{0\}$, where $\D^{\bot}$ is the orthogonal distribution with respect to $h$;
		\item The nonholonomic system $(L_{h},\D)$ is regular.
	\end{enumerate}
\end{theorem}

\begin{proof}
	Suppose first that $\D \cap \D^{\bot}=\{0\}$ and take $X_{u}\in T_{u}\D \cap F_{u}$, where $F$ is the distribution along $\D$ defined by
	\begin{equation*}
	F =\sharp_{\omega_{L_{h}}}( S^{*}(T\D^{o}))=\sharp_{\omega_{L_{h}}}(\tau_{Q}^{*}\D^{o}).
	\end{equation*}
	Since $X_{u}\in F_{u}$, then there exists $\alpha\in \D^{o}$ such that
	\begin{equation*}
	X_{u}=\sharp_{\omega_{L_{h}}}(\alpha^{v}_{u})=-(\sharp_{h}(\alpha))^{v}_{u}\in T_{u}\D,
	\end{equation*}
	where the last equation follows from Lemma \ref{omegametriclemma}.
	
	Therefore, $\sharp_{h}(\alpha)$ is in $\D$, but since $\sharp_{h}(\D^{o})=\D^{\bot}$, it must be the zero vector. Hence $X_{u}=-(\sharp_{h}(\alpha))^{v}_{u}=0$.
	
	Conversely, if $u\in \D \cap \D^{\bot}$, there exists $\alpha\in \D^{o}$ such that $u=\sharp_{h}(\alpha)$. Since the vector
	\begin{equation*}
	(\sharp_{h}(\alpha))^{v}_{u}\in T_{u}\D \cap F_{u}=\{0\},
	\end{equation*}
	then $\sharp_{h}(\alpha)=0$ or, in other words, $u=0$.
\end{proof}

By the theorem above, if the distribution $\D$ is non-degenerate with respect to $h$ then the nonholonomic system $(L_{h},\D)$ is regular and we can consider the nonholonomic SODE $\Gamma_{(L_{h},\D)}$ on $\D$. We will see now that the trajectories of $\Gamma_{(L_{h},\D)}$ are just the geodesics of the associated nonholonomic connection $\nabla^{nh}$ with initial velocities in $\D$.

\begin{theorem}\label{nhgeodesicth}
	Let $h$ be a pseudo-Riemannian metric and $L_{h}:TQ\rightarrow\R$ its associated Lagrangian. If $\D$ is a nonholonomic distribution satisfying $\D\cap\D^{\bot}=\{0\}$, then the base integral curves of $\Gamma_{(L_{h},\D)}$ are the geodesics of the connection $\nabla^{nh}$ with initial velocities in $\D$.
\end{theorem}

\begin{proof}
	Let $c_{v}:I\rightarrow Q$ be a trajectory of $\Gamma_{(L_{h},\D)}$ with initial velocity $\dot{c}_{v}(0)=v\in\D$. We must prove that $$\nabla^{nh}_{\dot{c_{v}}}\dot{c_{v}}=0.$$
	
	Given any $X\in\Gamma(\D)$, we will apply the geometric equation which defines $\Gamma_{(L_{h},\D)}$ {to the complete lift $X^{c}$ of $X$ (see Appendix A)} at points in $\D$. In fact, given $u\in\D$ and $\mu\in (TD)^{o}$
	\begin{equation*}
	\langle i_{\Gamma_{(L_{h},\D)}}\omega_{L_{h}}-dL_{h}(u), X^{c}(u) \rangle =\langle S^{*}(\mu) (u), X^{c}(u) \rangle.
	\end{equation*}
	Using the skew-symmetry of $\omega_{L_{h}}$ and the fact that $SX^{c}=X^{v}$ (see Appendix A) we get
	\begin{equation*}
	-\langle i_{X^{c}}\omega_{L_{h}},\Gamma_{(L_{h},\D)}\rangle -X^{c}(L_{h}) = \langle \mu, X^{v} \rangle.
	\end{equation*}
	Note that the right-hand side vanishes because $X^{v}\in\mathfrak{X}(\D)$ and $\mu\in (TD)^{o}$. Also, using Lemma \ref{cvlift} and equation \eqref{omegalift} from Appendix B on the left-hand side of the previous equations we deduce
	\begin{equation*}
	-d(\widehat{\flat_{h}(X)})(\Gamma_{(L_{h},\D)})+2\theta_{L_{(\nabla^{h} X)}}(\Gamma_{(L_{h},\D)})-L_{\Le_{X}h} = 0,
	\end{equation*}
	where $(\nabla^{h} X)$ is the $(0,2)$-tensor field defined by
	\begin{equation*}
		(\nabla^{h} X)(Y,Z)=h(\nabla^{h}_{X}Y,Z),
	\end{equation*}
	$\nabla^{h}$ the Levi-Civita connection associated with $h$ and $\widehat{\flat_{h}(X)}$ is the fiberwise linear function on $TQ$ induced by the 1-form $\flat_{h}(X)$.
	
	By the observation following Lemma \ref{cvlift}, we deduce that $L_{\Le_{X} h}=2L_{(\nabla^{h} X)}$. Moreover, since the vector field $\Gamma_{(L_{h},\D)}$ is a SODE along $\D$, it follows that $S\Gamma_{(L_{h},\D)}=\Delta|_{\D}$, with $\Delta$ being the Liouville vector field on $TQ$ and
	\begin{equation*}
	\theta_{L_{(\nabla^{h} X)}}(\Gamma_{(L_{h},\D)})=\Delta(L_{(\nabla^{h} X)})|_{\D}=2L_{(\nabla^{h} X)}|_{\D}.
	\end{equation*}
	So the equation boils down to
	\begin{equation*}
	\Gamma_{(L_{h},\D)}(u)(\widehat{\flat_{h}(X)})=2L_{\nabla_{X}}(u)=h(\nabla^{h}_u X, u).
	\end{equation*}
	Evaluating the last equation over the curve $\dot{c}_{v}$ and noting that $\Gamma_{(L_{h},\D)}(\dot{c}_{v})$ is just $\ddot{c}_{v}$, we deduce
	\begin{equation*}
	\ddot{c}_{v}\left( \widehat{\flat_{h}(X)} \right)=	h(\nabla^{h}_{\dot{c}_{v}}X,\dot{c}_{v}).
	\end{equation*}
	Then, of course,
	\begin{equation*}
	\frac{d}{dt}\left( \widehat{\flat_{h}(X)}(\dot{c}_{v}) \right)=	h(\nabla^{h}_{\dot{c}_{v}}X,\dot{c}_{v}),
	\end{equation*}
	which is by definition
	\begin{equation*}
	\frac{d}{dt}\left( h(X\circ c_{v},\dot{c}_{v}) \right)= h(\nabla^{h}_{\dot{c}_{v}}X,\dot{c}_{v}).
	\end{equation*}
	Using the fact that the connection is compatible with the metric, the previous equation reduces to
	\begin{equation*}
	h(\nabla^{h}_{\dot{c}_{v}}X,\dot{c}_{v})+h(X\circ c_{v}, \nabla^{h}_{\dot{c}_{v}}\dot{c}_{v}) =h(\nabla^{h}_{\dot{c}_{v}}X,\dot{c}_{v})
	\end{equation*}
	where the first term on the left-hand side cancels with the term on the right-hand side, giving
	\begin{equation*}
	h(X\circ c_{v}, \nabla^{h}_{\dot{c}_{v}}\dot{c}_{v})= 0.
	\end{equation*}
	Since $X$ is an arbitrary section in $\Gamma(\D)$ we conclude that $P(\nabla^{h}_{\dot{c}_{v}}\dot{c}_{v})=0$. But, since $\dot{c}_{v}\in\D$, the connection is forced to satisfy $P(\nabla^{h}_{\dot{c}_{v}}\dot{c}_{v})=\nabla^{nh}_{\dot{c}_{v}}\dot{c}_{v}$. Hence, we conclude
	\begin{equation*}
	\nabla^{nh}_{\dot{c}_{v}}\dot{c}_{v}=0.
	\end{equation*}
\end{proof}

Using Theorem \ref{nhgeodesicth}, we will describe the action of the nonholonomic SODE $\Gamma_{(L_{h}, \D)}$ on basic and fiberwise linear functions on $\D$.

Note that a basic function on $\D$ is of the form $f\circ \tau_{\D}$, with $f\in C^{\infty}(Q)$ and $\tau_{\D}:\D\rightarrow Q$ the vector bundle projection. On the other hand, a fiberwise linear function on $\D$ is given by $\widehat{\alpha}$, with $\alpha\in \Gamma(\D^{*})$ and
\begin{equation*}
	\widehat{\alpha}(v)=\langle \alpha(\tau_{\D}(v)) , v \rangle,\quad v\in\D.
\end{equation*}
In addition, a fiberwise quadratic function on $\D$ has the form $T^{\mathfrak{q}}$, with $T$ a section of the vector bundle $\D^{*}\otimes\D^{*}\rightarrow Q$ and
\begin{equation}\label{Tquadratic}
	T^{\mathfrak{q}}(v)=T(v,v), \quad v\in \D.
\end{equation}

\begin{remark}
	If $U$ is an open subset $U$ of $Q$ with local coordinates $(q^{i})$, $\{e_{a}\}$ is a local basis of sections of $\Gamma(\D)$, $\{e^{a}\}$ the dual basis of $\Gamma(\D^{*})$ and 
	\begin{equation*}
	\alpha=\alpha_{a}(q)e^{a}, \quad T=T_{ab}(q)e^{a}\otimes e^{b},
	\end{equation*}
	then
	\begin{equation*}
	\widehat{\alpha}(q^{i},v^{a})=\alpha_{b}(q)v^{b}, \quad T^{\mathfrak{q}}(q^{i},v^{a})=T_{ab}(q)v^{a}v^{b},
	\end{equation*}
	where $(q^{i},v^{a})$ are the local coordinates in $\D$ induced by the local coordinates $(q^{i})$ on $Q$ and the local basis of sections of $\Gamma(\D)$.
\end{remark}

\begin{theorem}\label{Gammanh}
	Let $h$ be a pseudo-Riemannian metric and $\D$ be a distribution in the same conditions as in the previous theorem. If $\Gamma_{(L_{h},\D)}$ is the nonholonomic SODE associated to the problem then it acts on basic functions and on fiberwise linear functions on $\D$ in the following way
	\begin{equation}
	\Gamma_{(L_{h},\D)}(f\circ\tau_{\D})=\widehat{df}|_{\D}, \quad \Gamma_{(L_{h},\D)}(\widehat{\alpha})=(\nabla^{nh}\alpha)^{\mathfrak{q}},
	\end{equation}
	for $f\in C^{\infty}(Q)$ and $\alpha\in\Gamma(\D^{*})$, where $\nabla^{nh}$ is the nonholonomic connection and $\nabla^{nh}\alpha$ is the section of the vector bundle $\D^{*}\otimes\D^{*}\rightarrow Q$ given by
	\begin{equation*}
		(\nabla^{nh}\alpha) (X,Y)=(\nabla^{nh}_{X}\alpha) (Y)=X(\alpha(Y))-\alpha(\nabla^{nh}_X Y), \ \text{for} \ X,Y\in\Gamma(\D).
	\end{equation*}
\end{theorem}

\begin{proof}
	Take $f\in C^{\infty}(Q)$ and $v\in \D$. Evaluating the vector field $\Gamma_{(L_{h},\D)}$ at $v$ and then applying it to the basic function $f\circ \tau_{\D}$ is equivalent to apply the vector $T\tau_{\D}(\Gamma_{(L_{h},\D)}(v))$ to the function $f$.
	
	Since $\Gamma_{(L_{h},\D)}$ is a SODE on $\D$, its projection to the tangent bundle $TQ$ is the identity on $\D$. Therefore, we obtain
	\begin{equation*}
		\Gamma_{(L_{h},\D)}(v)(f\circ \tau_{\D})=v(f),
	\end{equation*}
	which is exactly $\widehat{df}|_{\D}(v)$.
	
	As for the second expression, let $\alpha$ be a section of $\D^{*}$ and take a base integral curve of $\Gamma_{(L_{h},\D)}$ denoted by $c_{v}:I\rightarrow Q$, where the subscript means that $\dot{c}_{v}(0)=v$.
	
	Let $\sharp_{h}(\alpha):Q\rightarrow \D$ be the section of $\D$ given by
	\begin{equation*}
		h(\sharp_{h}(\alpha), X)=\alpha(X), \quad \forall X\in \Gamma(\D).
	\end{equation*}
	
	Applying $\Gamma_{(L_{h},\D)}(v)$ to the fiberwise linear function $\widehat{\alpha}$ is equivalent to
	\begin{equation*}
		\Gamma_{(L_{h},\D)}(v)(\widehat{\alpha})=\ddot{c}_{v}(0) (\widehat{\alpha}).
	\end{equation*}
	Using the definition of derivative along a curve, the last line is equivalent to
	\begin{equation*}
	\Gamma_{(L_{h},\D)}(v)(\widehat{\alpha})=\left. \frac{d}{dt} \right|_{t=0} (\widehat{\alpha}\circ \dot{c}_{v}(t)),
	\end{equation*}
	and thus, using the notation we have just introduced we write
	\begin{equation*}
	\Gamma_{(L_{h},\D)}(v)(\widehat{\alpha})=\left. \frac{d}{dt} \right|_{t=0} h(\sharp_{h}(\alpha)\circ c_{v}(t),\dot{c}_{v}(t)),
	\end{equation*}
	Using the compatibility condition \eqref{Dcompatibility}, this is equivalent to
	\begin{equation*}
	\Gamma_{(L_{h},\D)}(v)(\widehat{\alpha})=h((\nabla_{v}^{nh}\sharp_{h}(\alpha))( c_{v}(0)),v)+h(\sharp_{h}(\alpha)(c_{v}(0)),\nabla_{v}^{nh}\dot{c}_{v}(0)).
	\end{equation*}
	Since by Theorem \ref{nhgeodesicth}, $c_{v}$ is a geodesic of the connection $\nabla^{nh}$, the last term above vanishes.
	
	Suppose now that $X$ is a section of $\D$ extending $v$, i.e., $X(q)=v$. With this new ingredient the last equation may be rewritten as
	\begin{equation*}
	\Gamma_{(L_{h},\D)}(v)(\widehat{\alpha})=h(\nabla_{X}^{nh}\sharp_{h}\alpha (q),X (q)).
	\end{equation*}
	By adding and subtracting the term $h(\sharp_{h}\alpha(q), \nabla_{X}^{nh} X (q))$ in the previous equation we may apply \eqref{Dcompatibility} and get
	\begin{equation*}
	\Gamma_{(L_{h},\D)}(v)(\widehat{\alpha})=X(q)h(\sharp_{h}\alpha,X )-h(\sharp_{h}\alpha (q), \nabla_{X}^{nh} X(q)),
	\end{equation*}
	and finally unyielding the definition of $\sharp_{g}(\alpha)$ we deduce
	\begin{equation*}
	\Gamma_{(L_{h},\D)}(v)(\widehat{\alpha})=[(\nabla_{X}^{nh} \alpha) (X)] (q).
	\end{equation*}
	The right-hand side of the last equation is a $(0,2)$-tensor, as such, its value does not depend on the whole section and thus $\nabla^{nh}\alpha (v,v)$ is well-defined. Therefore, using the notation introduced before the theorem, it can be rewritten as $(\nabla^{nh}\alpha)^{\mathfrak{q}}(v)$.
\end{proof}


\section{Nonholonomic Jacobi fields}\label{non-holonomic-Jacobi-fields}

\subsection{Definition and some examples}

First of all, we will introduce the notion of a Jacobi field for a nonholonomic system as an extension of the definition of a Jacobi field (over a geodesic) for a Riemannian metric.

\begin{definition}\label{definition31}
	Let $(L,\D)$ be a nonholonomic system with configuration manifold $Q$. A vector field $W:I\rightarrow TQ$ along a curve $c:I\rightarrow Q$ is said to be a \textit{nonholonomic Jacobi field} for the system $(L,\D)$ if it is the infinitesimal variation vector field of a family of nonholonomic trajectories of $(L,\D)$.
\end{definition}
 
 So, according to the definition,
 \begin{equation*}
	 W(t)=\left. \frac{\partial}{\partial s}\right|_{s=0}(\tau_{Q}\circ \Phi)(s,t),
 \end{equation*}
 where
 \begin{equation*}
 	\begin{split}
	 	\Phi:(-\varepsilon,\varepsilon)\times I & \rightarrow \D \\
	 	(s,t) & \mapsto \Phi_{s}(t)
 	\end{split}
 \end{equation*}
is a smooth map and, for each $s\in (-\varepsilon,\varepsilon)$, $\Phi_{s}:I\rightarrow \D$ is the tangent lift $\dot{c}_{s}:I\rightarrow \D$ of a trajectory $c_{s}:I\rightarrow Q$ of $(L,\D)$ with $c_{0}=c$.

We remark that, in general, $W$ its not a section of $\D$. Its value may well assume any vector in $TQ$ (see the next example \ref{nhparjacobi}).

In what follows, we will assume that $(L,\D)$ is regular. In fact, we will assume that $L=L_{g}$ with $g$ a Riemannian metric. Thus, we can consider the nonholonomic SODE $\Gamma_{(L_{g},\D)}$ and
\begin{equation*}
	\Phi_{s}(t)=\phi_{t}^{\Gamma_{(L_{g},\D)}}(v(s)),
\end{equation*}
where $\phi_{t}^{\Gamma_{(L,\D)}}$ is the local flow of $\Gamma_{(L_{g},\D)}$ and $v:(-\varepsilon,\varepsilon)\rightarrow\D$ is a curve on $\D$. Therefore, a Jacobi field could be written as
\begin{equation*}
W(t)=\left. \frac{\partial}{\partial s}\right|_{s=0}\left(\tau_{Q}\circ \phi_{t}^{\Gamma_{(L_{g},\D)}}(v(s))\right),
\end{equation*}

\begin{remark}
	If the system is unconstrained, that is  $\D=TQ$, then it is clear that $W:I\rightarrow TQ$ is a Jacobi field for the system $(L_{g},TQ)$ if and only if $W$ is a Jacobi field for the Riemannian metric $g$ on $Q$ (see, for instance, \citep*{ON83}).
\end{remark}

Next, we will present a method that allows us to obtain, under certain conditions, nonholonomic Jacobi fields.
 
 \begin{theorem}\label{jacobikilling}
 	Let $(L_{g},\D)$ be a purely kinematic nonholonomic system on the manifold $Q$ associated with the Riemannian metric $g$, $c_{v}:I\rightarrow Q$ a nonholonomic solution and let $W\in\mathfrak{X}(Q)$  be a vector field satisfying the following three conditions:
 	\begin{enumerate}
 		\item[(i)] $[W,\Gamma(\D)]\subseteq \Gamma(\D)$ ;
 		\item[(ii)] $\left. \mathcal{L}_{W}g \right|_{\Gamma(\D)\times\Gamma(\D)}=0$;
 		\item[(iii)] $\left. \mathcal{L}_{W}g \right|_{[\Gamma(\D),\Gamma(\D)]\times\Gamma(\D)}=0$;
 	\end{enumerate}
 	then $W\circ c_{v}:I\longrightarrow TQ$ is a Jacobi field along the nonholonomic solution $c_{v}$.
 \end{theorem}

\begin{proof}
	Let us first show that the vector field $\left.W^{c}\right|_{\D}\in\mathfrak{X}(\D)$, which is clearly equivalent to having its flow $T\phi_{t}^{W}$ contained in $\D$. Given $\alpha\in\Gamma(\D^{o})$, its associated fiberwise linear function $\widehat{\alpha}\in C^{\infty}(TQ)$ vanishes on $\D$. In fact, 
	\begin{equation*}
		\D=\{ v\in TQ \ | \ \widehat{\alpha}(v)=0, \ \forall \ \alpha\in \Gamma(\D^{o}) \}.
	\end{equation*}
	Therefore, it is enough to show that $W^{c}(\widehat{\alpha})|_{\D}=0$, for $\alpha\in\Gamma(\D^{o})$. Let $X\in\Gamma(\D)$ then
	\begin{equation*}
		W^{c}(\widehat{\alpha}) \circ X=\widehat{\Le_{W}\alpha} \circ X,
	\end{equation*}
	using the definition of complete lift (see equation \eqref{completelift} in Appendix A). Applying now the characterization of the Lie derivative of a one-form we deduce
	\begin{equation*}
		W^{c}(\widehat{\alpha}) \circ X=W(\alpha(X))-\alpha([W,X]),
	\end{equation*}
	The first term vanishes identically since $\alpha$ is a section of the annihilator of $\D$ and $X$ is a section of $\D$ while the second one vanishes since $[W,X]$ is a section of $\D$, by the first hypothesis in the statement of the theorem. Hence, since $X$ was arbitrary, we deduce that $W^{c}|_{\D}\in \mathfrak{X}(\D)$.
	
	Now, if the vector fields $W^{c}|_{\D}$ and $\Gamma_{(L_{g},\D)}$ commute, then their flows $T\phi_{s}^{W}$ and $\phi_{t}^{\Gamma_{(L_{g},\D)}}$, respectively, also commute. Take $v\in\D$ and project the composition of the flows to $Q$ using the bundle projection $\tau_{\D}:\D\rightarrow Q$. Then
	\begin{equation*}
		\left(\tau_{\D}\circ T\phi_{s}^{W}\circ \phi_{t}^{\Gamma_{(L_{g},\D)}}\right)(v) =\left(\tau_{\D}\circ \phi_{t}^{\Gamma_{(L_{g},\D)}}\circ T\phi_{s}^{W}\right)(v).
	\end{equation*}
	Since the tangent lift of the flow of $W$ is a vector bundle isomorphism over $\phi_{s}^{W}$ and since the projection $\tau_{\D}(\phi_{t}^{\Gamma_{L_{g},\D}}(v))$ of $\phi_{t}^{\Gamma_{L_{g},\D}}(v)$ is just the trajectory of $\Gamma_{(L_{g},\D)}$ with initial velocity $v\in\D$, which we denote in general by $c_{v}$, we find
	\begin{equation*}
		\left(\phi_{s}^{W}\circ \tau_{\D}\circ \phi_{t}^{\Gamma_{(L_{g},\D)}}\right)(v) =c_{T\phi_{s}^{W}(v)}(t).
	\end{equation*}
	And applying similar considerations again, the last line reduces to
	\begin{equation*}		
		\phi_{s}^{W}\circ c_{v}(t) =c_{T\phi_{s}^{W}(v)}(t).
	\end{equation*}
	This computation proves that the 2-parameter family
	\begin{equation}
		\Phi:(t,s)\mapsto \phi_{s}^{W}\circ c_{v}(t)
	\end{equation}
	is actually a variation by trajectories of $\Gamma_{(L_{g},\D)}$. Moreover, its infinitesimal variation vector field is given by
	\begin{equation*}
		\left. \frac{d \Phi}{d s}\right|_{s=0}(t)=\left. \frac{d}{d s}\right|_{s=0}\phi_{s}^{W}\circ c_{v}(t)=W(c_{v}(t)).
	\end{equation*}
	Therefore $W\circ c_{v}:I\rightarrow TQ$ is a Jacobi field along $c_{v}$.
	
	So, all we need to show is that $W^{c}|_{\D}$ and $\Gamma_{(L_{g},\D)}$ commute. We will prove this result in the next proposition.

	\begin{proposition}
		If $(L_{g},\D)$ is a nonholonomic system on $Q$ and $W$ is a vector field on $Q$ in the same conditions as in Theorem \ref{jacobikilling}, then we have that
		\begin{equation*}
			[W^{c}|_{\D},\Gamma_{(L_{g},\D)}]=0.
		\end{equation*}
	\end{proposition}
	
	\begin{proof}
	We will prove the proposition by computing the action of $[W^{c}|_{\D},\Gamma_{L_{g},\D}]$ on basic and fiberwise linear functions in $\C^{\infty}(\D)$, which are generated by functions of the type $f\circ\tau_{\D}$ and $\widehat{\alpha}$, with $f\in\C^{\infty}(Q)$ and $\alpha\in\Gamma(\D^{*})$. We have that
	\begin{equation*}
		[W^{c}|_{\D},\Gamma_{(L_{g},\D)}](f\circ\tau_{Q})=W^{c}|_{\D}(\Gamma_{(L_{g},\D)}(f\circ\tau_{Q}))-\Gamma_{(L_{g},\D)}(W^{c}|_{\D}(f\circ\tau_{Q})).
	\end{equation*}
	
	Using equation \eqref{completelift} in Appendix A and Theorem \ref{Gammanh} the last line becomes
	\begin{equation*}
		\begin{split}
			[W^{c}|_{\D},\Gamma_{(L_{g},\D)}](f\circ\tau_{\D}) & =W^{c}|_{\D}(\widehat{df}|_{\D})-\Gamma_{(L_{g},\D)}(W(f)\circ\tau_{\D}) \\
			& = \widehat{\Le_W df}|_{\D}-\widehat{d(W(f))}|_{\D}=0.
		\end{split}
	\end{equation*}
	
	On the other hand, the actions over functions $\widehat{\alpha}$ with $\alpha\in\Gamma(D^{*})$ is given by
	\begin{equation*}
	[W^{c}|_{\D},\Gamma_{(L_{g},\D)}](\widehat{\alpha})=W^{c}|_{\D}(\Gamma_{(L_{g},\D)}(\widehat{\alpha}))-\Gamma_{(L_{g},\D)}(W^{c}|_{\D}(\widehat{\alpha})).
	\end{equation*}
	It is a simple computation to show that on $\D$
	\begin{equation}
	\widehat{\alpha}=\widehat{P^{*}\alpha}|_{\D}, \quad (\nabla^{nh}\alpha)^{\mathfrak{q}}=(\nabla^{nh}P^{*}\alpha)^{\mathfrak{q}}|_{\D}=(\nabla^{g}P^{*}\alpha)^{\mathfrak{q}}|_{\D},
	\end{equation}
	where $P:TQ\rightarrow \D$ is the orthogonal projector, $\nabla^{nh}$ is the nonholonomic connection and $\nabla^{g}$ is the Levi-Civita connection with respect to $g$. In the expression above, we extended the notation for fiberwise quadratic functions we introduced before. Indeed, given any vector bundle $V\rightarrow Q$ and a section $T$ of $V^{*}\otimes V^{*}\rightarrow Q$, then $T^{\mathfrak{q}}$ is the fiberwise quadratic function on $V$ induced by $T$.
		
	Hence, using again equation \eqref{completelift} and Theorem \ref{Gammanh} we have that
	\begin{equation}\label{commutator}
		[W^{c}|_{\D},\Gamma_{(L_{g},\D)}](\widehat{\alpha})=(\Le_{W}(\nabla^{g}P^{*}\alpha))^{\mathfrak{q}}|_{\D}-(\nabla^{nh}\Le_{W} (P^{*}\alpha)|_{\D})^{\mathfrak{q}},
	\end{equation}
	where we have also used equation \eqref{completequadratic} on the first term of the right-hand side. Both terms appearing above are fiberwise quadratic functions associated to $(0,2)$-tensors (see equation \eqref{completequadratic} in Appendix A). Given $X\in\Gamma(\D)$, the first term reduces to
	\begin{equation*}
		\Le_{W}(\nabla^{g}P^{*}\alpha) (X,X)=W(\nabla^{g}_{X}P^{*}\alpha (X))-\nabla^{g}_{[W,X]}P^{*}\alpha (X)-\nabla^{g}_{X}P^{*}\alpha ([W,X]).
	\end{equation*}
	Note that $\exists$ $Y\in \Gamma(\D)$ such that $P^{*}\alpha=\flat_{g}(Y)$. So we can rewrite the expression above in terms of the vector field $Y$. Moreover, using the identity
	\begin{equation}
		\flat_{g}(\nabla^{g}_{X} Y)=\nabla^{g}_{X}\flat_{g}(Y), \quad X,Y\in\mathfrak{X}(Q),
	\end{equation}
	we get
	\begin{equation*}
		\Le_{W}(\nabla^{g}P^{*}\alpha) (X,X)=W(\flat_{g}(\nabla^{g}_{X}Y) (X))-\flat_{g}(\nabla^{g}_{[W,X]}Y) (X)-\flat_{g}(\nabla^{g}_{X}Y) ([W,X]).
	\end{equation*}
	Now we use Lemma \ref{liedelemmma} in Appendix B to reduce the previous to
	\begin{equation*}
		\begin{split}
			\Le_{W}(\nabla^{g}P^{*}\alpha) (X,X) & =\frac{1}{2}W(\Le_{Y}g(X,X))-\frac{1}{2}(\Le_{Y}g([W,X],X)\\
			& -d(\flat_{g}(Y))([W,X],X))-\frac{1}{2}(\Le_{Y}g(X,[W,X]) \\
			& -d(\flat_{g}(Y))(X,[W,X])) \\
			& = \frac{1}{2}W(\Le_{Y}g(X,X))-\Le_{Y}g([W,X],X) \\
			& = \frac{1}{2}\Le_{W}(\Le_{Y}g)(X,X)
		\end{split}
	\end{equation*}
	But, one can prove that for a $(0,2)$-tensor $g$ and any $X, Y, Z, Z' \in\mathfrak{X}(Q)$ we have
	\begin{equation}
		\Le_{[X,Y]}g (Z,Z')=\Le_{X}(\Le_{Y}g)(Z,Z')-\Le_{Y}(\Le_{X}g)(Z,Z').
	\end{equation}
	Hence, using this fact, we conclude that
	\begin{equation*}
	\begin{split}
		\Le_{W}(\nabla^{g}P^{*}\alpha) (X,X) & =\frac{1}{2}\Le_{[W,Y]}g (X,X)+\frac{1}{2}\Le_{Y}(\Le_{W}g)(X,X) \\
		 & = g(\nabla^{g}_{X}[W,Y],X),
	\end{split}
	\end{equation*}
	where
	\begin{equation*}
		\frac{1}{2}\Le_{Y}(\Le_{W}g)(X,X)=\frac{1}{2}Y((\Le_{W}g)(X,X))-(\Le_{W}g)([Y,X],X)
	\end{equation*}
	vanishes because $W$ satisfies hypothesis $(ii)$ and $(iii)$. Thus,
	\begin{equation*}
		\left(\Le_{W}(\nabla^{g}P^{*}\alpha)\right) (X,X)=g(\nabla^{g}_{X}[W,Y],X).
	\end{equation*}
	As for the second term in \eqref{commutator}, we proceed by unyielding the definitions
	\begin{equation*}
		\nabla^{nh}\Le_{W} P^{*}\alpha (X,X)=\nabla^{nh}_{X}(\Le_{W} \flat_{g}(Y)) X=X(\Le_{W} \flat_{g}(Y) (X))-\Le_{W} \flat_{g}(Y)(\nabla^{nh}_{X} X).
	\end{equation*}
	For any $Z\in\Gamma(\D)$ one has that
	\begin{equation}
		\Le_{W} \flat_{g}(Y) (Z)=\flat_{g}([W,Y])(Z)+(\Le_{W}g)(Y,Z)=\flat_{g}([W,Y])(Z).
	\end{equation}
	Therefore,
	\begin{equation*}
		\begin{split}
			\nabla^{nh}\Le_{W} P^{*}\alpha (X,X) & =X(\flat_{g}([W,Y])(X))-\flat_{g}([W,Y])(\nabla^{nh}_{X} X) \\
			 & = X(g([W,Y],X))-g([W,Y],\nabla^{nh}_{X} X) \\
			 & = X(g([W,Y],X))-g([W,Y],P\nabla^{g}_{X} X).
		\end{split}
	\end{equation*}
	So, using that $[W,Y]\in\Gamma(\D)$, it follows that
	\begin{equation*}
		\begin{split}
			 \nabla^{nh}\Le_{W} P^{*}\alpha (X,X) & = X(g([W,Y],X))-g([W,Y],\nabla^{g}_{X} X) \\
			 & = g(\nabla^{g}_{X}[W,Y],X).
		\end{split}
	\end{equation*}
	Hence both terms in equation \eqref{commutator} cancel and	$[W^{c}|_{\D},\Gamma_{L_{g},\D}](\widehat{\alpha})=0$.
	\end{proof}
\end{proof}

From Theorem \ref{jacobikilling}, it follows that
\begin{corollary}\label{killingcorollary}
	Let $(L_{g},\D)$ be a purely kinematic nonholonomic system on the manifold and $c_{v}:I\rightarrow Q$ a nonholonomic solution with initial velocity $v\in\D$. If $W$ is an infinitesimal symmetry of the system $(L_{g},\D)$, that is, $W$ is a Killing vector field for the Riemannian metric $g$ (i.e., $\Le _{W} g=0$) and an infinitesimal symmetry of $\D$ (that is $[W,\Gamma(\D)]\subseteq \Gamma(\D)$) then $W\circ c_{v}:I\rightarrow TQ$ is a nonholonomic Jacobi field for the system $(L_{g},\D)$.
\end{corollary}

\begin{remark}
	If the system $(L_{g},\D)$ is unconstrained (that is, $\D=TQ$), then using Corollary \ref{killingcorollary}, we recover a well-known result in Riemannian geometry (see, for example, Lemma 26, Chapter 9 in \citep*{ON83}): the restriction of a Killing vector field to a geodesic is a Jacobi field for the Riemannian metric.
\end{remark}

\begin{example}\label{nhparjacobi}
	We show, by applying the previous corollary, that the vector field $W=\frac{\partial}{\partial z}$ is a Jacobi field for the nonholonomic particle, along any nonholonomic solution.
	
	It is clear that the first condition in the theorem is satisfied, since the vector field $\frac{\partial}{\partial z}$ commutes with the vector fields $e_{1}=\frac{\partial}{\partial x}+y\frac{\partial}{\partial z}$ and $e_{2}=\frac{\partial}{\partial y}$ generating the module of sections $\Gamma(\D)$.
	
	On the other hand, $\frac{\partial}{\partial z}$ is a Killing vector field for the euclidean metric $g$ on $\R^{3}$, so it satisfies the hypothesis in Corollary \ref{killingcorollary}.
	
	Therefore, by Corollary \ref{killingcorollary}, the vector field  $\frac{\partial}{\partial z}$ is a Jacobi field along any trajectory of the nonholonomic system $(L_{g},\D)$.
\end{example}

\begin{example}
	A more physical example is the vertical rolling disk, which models the motion of a rolling penny on a plane. It is a nonholonomic system with a Lagrangian function of kinetic type given by $L_{g}:T(\R^{2}\times \Es^{1}\times \Es^{1})\rightarrow \R$, with
	\begin{equation*}
		L_{g}(x,y,\theta,\varphi,\dot{x},\dot{y},\dot{\theta},\dot{\varphi})=\frac{1}{2}(\dot{x}^{2}+\dot{y}^{2}+I\dot{\theta}^{2}+J\dot{\varphi}^{2}),
	\end{equation*}
	where $I$ and $J$ are real numbers known as moment of inertia and nonholonomic constraints imposed by the equations
	\begin{equation*}
		\dot{x}=R \dot{\theta}\cos (\varphi), \quad \dot{y}=R \dot{\theta}\sin (\varphi),
	\end{equation*}
	 where $R$ is the radius of the disk. For more details see \citep*{Bloch}. Now, it is easy to see that the constraints form a distribution $\D$ with $\text{rank}(\D)=2$ generated by the vector fields
	 \begin{equation*}
	 	e_{1}=R \cos (\varphi)\frac{\partial}{\partial x} + R \sin (\varphi)\frac{\partial}{\partial y}+\frac{\partial}{\partial \theta}, \quad e_{2}=\frac{\partial}{\partial \varphi}.
	 \end{equation*}
	 It is easy to see that the vector field $W=\frac{\partial}{\partial \theta}$ is an infinitesimal symmetry of $\D$. Moreover, $W$ is a Killing vector field for the Riemannian metric $g$ on $\R^{2}\times \Es^{1}\times \Es^{1}$ associated to the Lagrangian $L_{g}$.
\end{example}

\begin{example}
	We will consider again the nonholonomic particle. However, now we will consider a counter-example of a Jacobi field which is not a Killing vector field for $g$ and another one of a Jacobi field which is not a symmetry of the distribution.
	
	Let $c_{v(s)}:I\rightarrow \R^{3}$ be a trajectory of the nonholonomic particle with $c_{v}(0)=(x_{0},y_{0},z_{0})$ and initial velocity $v(s)=(\dot{x}_{0}(s),\dot{y}_{0}(s),y_{0}\dot{x}_{0}(s))$ for each $s\in(-\varepsilon,\varepsilon)$.
	
	On one hand, suppose that $\dot{y}_{0}(s)\equiv 0$ and so the trajectory has the local expression
	\begin{equation}\label{nhsolution0}
	\begin{cases}
	x_{s}(t)= \dot{x}_0(s) t+x_0\\
	y_{s}(t)=y_0 \\
	z_{s}(t)= y_0 \dot{x}_0(s) t+z_0.
	\end{cases}
	\end{equation}
	The curve $W:I\rightarrow TQ$ defined by
	\begin{equation*}
	W(t)= \left. \frac{d}{ds} \right|_{s=0} c_{v(s)} = u t \cdot \left(\frac{\partial}{\partial x}+y_{0} \frac{\partial}{\partial z}\right)
	\end{equation*}
	is a Jacobi field along $c_{v(0)}$ by definition, where $u$ denotes $\frac{d}{ds}|_{s=0}\dot{x}_0(s)$. Supposing that $\dot{x}_0(0)$ is not zero then the vector field $\widetilde{W}\in\mathfrak{X}(\R^{3})$ defined by
	\begin{equation*}
	\widetilde{W}(x,y,z)=u \cdot\left(\frac{x-x_{0}}{\dot{x}_0(0)}\right)\cdot \left(\frac{\partial}{\partial x}+y \frac{\partial}{\partial z}\right)
	\end{equation*}
	extends $W(t)$ over the curve $c_{v(0)}$, that is,
	\begin{equation*}
	W(t)= (\widetilde{W}\circ c_{v(0)})(t).
	\end{equation*}
	However, as it is clear, $\widetilde{W}$ is not a symmetry of the distribution.
	
	On the other hand, suppose that $\dot{y}_{0}(s)$ does not vanish. Then the local expression of the trajectory is
	\begin{equation}\label{nhsolution}
	\begin{cases}
	x_{s}(t)= \frac{\dot{x}_0 (s)}{\dot{y}_0 (s)}\sqrt{y_0^2+1}(\arcsinh(\dot{y}_0 (s) t+y_0)-\arcsinh(y_0))+x_0\\
	y_{s}(t)=\dot{y}_0 (s) t+y_0 \\
	z_{s}(t)= \frac{\dot{x}_0 (s)}{\dot{y}_0 (s)}\sqrt{y_0^2+1}(\sqrt{(\dot{y}_0 (s) t+y_0)^2+1}-\sqrt{y_0^2+1})+z_0.
	\end{cases}
	\end{equation}
	Suppose that
	\begin{equation*}
	\left. \frac{d}{ds}\right|_{s=0}\dot{x}_0 (s)=u \quad \text{and} \quad \left.\frac{d}{ds}\right|_{s=0}\dot{y}_0 (s)=0.
	\end{equation*}
	Then the vector field $W:I\rightarrow TQ$ defined as before is a Jacobi field over the trajectory $c_{v(0)}$ and has the local expression
	\begin{equation*}
	\begin{split}
	W(t)=  \frac{u}{\dot{y}_0 (0)}\sqrt{y_0^2+1} \cdot & \left[ \left(\arcsinh(\dot{y}_0 (0) t+y_0)-\arcsinh(y_0)\right) \frac{\partial}{\partial x} \right. \\
	& \left. + \left( \sqrt{(\dot{y}_0 (0) t+y_0)^2+1}-\sqrt{y_0^2+1} \right) \frac{\partial}{\partial z} \right].
	\end{split}
	\end{equation*}
	Following the same construction as before, supposing that $\dot{x}_0 (0)$ does not vanish, then the vector field
	\begin{equation*}
	\widetilde{W}(x,y,z)=\frac{u}{\dot{x}_0(0)} \left((x-x_{0})\frac{\partial}{\partial x}+(z-z_{0}) \frac{\partial}{\partial z}\right)
	\end{equation*}
	extends $W(t)$, in the same sense than before. However, it is easy to check that
	\begin{equation}
	(\Le_{\widetilde{W}} g)\left(\frac{\partial}{\partial x},\frac{\partial}{\partial x}\right)=\frac{2 u}{\dot{x}_0(0)},
	\end{equation}
	hence $\widetilde{W}$ is not a Killing vector field for $g$.
\end{example}

\begin{example}
	Let us find a similar counterexample for the vertical rolling disk dynamics.
	
	Let $c_{v(s)}:I\rightarrow \R^{2}\times \Es^{1} \times \Es^{1}$ be a trajectory of the vertical rolling disk with $c_{v}(0)=(x_{0},y_{0},\theta_{0},\varphi_{0})$ and initial velocity in $\D$ given by $v(s)=(\dot{x}_{0}(s),\dot{y}_{0}(s),\Omega(s),\omega(s))$ for each $s\in(-\varepsilon,\varepsilon)$.
	
	The explicit solution of the nonholonomic dynamics is discussed in \citep*{Bloch}, where we find that
	\begin{equation}\label{angular:variables}
	\begin{cases}
	\theta_{s}(t)=\Omega(s)t+\theta_{0} \\
	\varphi_{s}(t)=\omega(s)t+\varphi_{0}
	\end{cases}
	\end{equation}
	and the expression for the variables $x$ and $y$ is determined by integrating the constraints.
	
	Suppose that $\omega(s)\equiv 0$, in which case the trajectory is given by the local expressions \eqref{angular:variables} and
	\begin{equation*}
	\begin{cases}
	x_{s}(t)=\Omega(s)tR \cos(\varphi_{0})+x_{0} \\
	y_{s}(t)=\Omega(s)tR\sin(\varphi_{0})+y_{0}.
	\end{cases}
	\end{equation*}
	Now, let
	\begin{equation*}
	\left. \frac{d}{ds} \right|_{s=0}\Omega(s)=u \quad \text{and} \quad \Omega(0)=\Omega_{0},
	\end{equation*}
	with $\Omega_{0}$ different from zero. Then the vector field $W:I\rightarrow TQ$ obtained by
	\begin{equation*}
	W(t)= \left. \frac{d}{ds} \right|_{s=0} c_{v(s)} = u t \cdot \left(R\cos(\varphi_{0}) \frac{\partial}{\partial x}+R\sin(\varphi_{0}) \frac{\partial}{\partial y}+ \frac{\partial}{\partial \theta}\right)
	\end{equation*}
	is a Jacobi field along $c_{v(0)}$ by definition.
	
	Moreover, the vector field $\widetilde{W}\in\mathfrak{X}(\R^{2}\times \Es^{1} \times \Es^{1})$ defined by
	\begin{equation*}
	\widetilde{W}(x,y,\theta,\varphi)=u \cdot\left(\frac{\theta-\theta_{0}}{\Omega_0}\right)\cdot \left(R\cos(\varphi_{0}) \frac{\partial}{\partial x}+R\sin(\varphi_{0}) \frac{\partial}{\partial y}+ \frac{\partial}{\partial \theta}\right)
	\end{equation*}
	extends $W(t)$ over the curve $c_{v(0)}$, that is,
	\begin{equation*}
	W(t)= (\widetilde{W}\circ c_{v(0)})(t).
	\end{equation*}
	
	However it is easy to see that $\widetilde{W}$ is not an infinitesimal symmetry of the distribution and it is not a Killing vector field with respect to the metric $g$.
\end{example}
 
 \subsection{The lift of the kinematic nonholonomic system and the nonholonomic Jacobi fields}

Denote by $g^{c}$ the complete lift of the Riemannian metric $g$ (see \eqref{def:complete:lift:g} in Appendix C). Then $g^{c}$ is a pseudo-Riemannian metric on $TQ$ and we may consider the Lagrangian function $L_{g^{c}}:TTQ\rightarrow \R$ associated with $g^{c}$. We recall that
\begin{equation*}
L_{g^{c}}=L_{g}^{c}\circ\kappa_{Q},
\end{equation*}
where $L_{g}^{c}$ is the complete lift of the Lagrangian function $L_{g}$ and $\kappa_{Q}:TTQ\rightarrow TTQ$ is the canonical involution of the double tangent bundle $TTQ$ (see Appendix C).

Now, consider the complete lift $\D^{c}$ of the distribution $\D$ as a distribution on $TQ$, whose space of sections is
\begin{equation*}
\Gamma(\D^{c})=\langle \{ X^{c},X^{v} | \ X\in\Gamma(\D)\} \rangle.
\end{equation*}
Here, $X^{c}$ and $X^{v}$ are the complete and vertical lifts of the vector field $X\in\Gamma(\D)$. The  distribution $\D^{c}$ was considered in \citep*{YaIs73}.

$\D^{c}$ is not only a vector subbundle (over $TQ$) of the vector bundle $\tau_{TQ}:TTQ\rightarrow TQ$ but also a vector bundle over $\D$ with vector bundle projection
$(T \tau_{Q})|_{\D^{c}}:\D^{c}\rightarrow\D$. In fact, if $X\in\Gamma(\D)$ then
\begin{equation*}
	(T\tau_{Q})(X^{c})=X\circ\tau_{Q}, \quad 	(T\tau_{Q})(X^{v})=0\circ\tau_{Q},
\end{equation*}
where $0:Q\rightarrow TQ$ is the zero section.

On the other hand, the tangent bundle $T\D$ to $\D$ is also a double vector bundle. Indeed, apart from the canonical vector bundle structure $\tau_{\D}:T\D\rightarrow\D$, it is also a vector bundle over $TQ$ with vector bundle projection $T(\tau_{Q}|_{\D}):T\D\rightarrow TQ$.

In addition, using that $\kappa_{Q}$ is an involution from the vector bundle $\tau_{TQ}:TTQ\rightarrow TQ$ to the vector bundle $T\tau_{Q}:TTQ \rightarrow TQ$ (see Appendix C), it follows that the restriction of $\kappa_{Q}$ to $\D^{c}\subseteq TTQ$ is also an isomorphism between the vector bundle $\tau_{\D^{c}}:\D^{c}\rightarrow TQ$ and $T(\tau_{Q}|_{\D}):T\D\rightarrow TQ$ (respectively, between $(T\tau_{Q})|_{\D^{c}}:\D^{c}\rightarrow\D$ and $\tau_{\D}:T\D\rightarrow\D$) over the identity of $TQ$ (respectively, over the identity of $\D$). The diagram in Figure \ref{Fig:Canonicalinv} illustrates the situation. Note that the inverse morphism of this double vector bundle isomorphism is $(\kappa_{Q})|_{T\D}:T\D\rightarrow \D^{c}$.
\begin{figure}[!htb]
	\begin{center}
		\begin{tikzcd}
		& TQ          &                                                                                                                              \\
		\mathcal{D}^{c} \arrow[ru, "\tau_{\mathcal{D}^{c}}"] \arrow[rr, "(\kappa_{Q})|_{\mathcal{D}^{c}}"] \arrow[rd, "(T\tau_{Q})|_{\mathcal{D}^{c}}"'] &             & T\mathcal{D} \arrow[lu, "T(\tau_{Q}|_{\mathcal{D}})"'] \arrow[ld, "\tau_{\mathcal{D}}"] \\
		& \mathcal{D} &                                                                                                                             
		\end{tikzcd}
	\end{center}
	\caption{Commutative diagram showing how the restriction of the canonical involution to $\D^{c}$ commutes with the projections to $\D$ and $TQ$.}
	\label{Fig:Canonicalinv}
\end{figure}

\begin{definition}
	The nonholonomic system $(L_{g^{c}},\D^{c})$ is the \textit{complete lift} of the nonholonomic system of kinetic type $(L_{g},\D)$.
\end{definition}

The aim of this section is to prove the following theorem:

\begin{theorem}\label{MainTheorem}
	Let $(L_{g},\D)$ be a nonholonomic system of kinetic type and $\Gamma_{(L_{g},\D)}$ the associated nonholonomic SODE. Then
	\begin{enumerate}
		\item[(i)] The complete lift $(L_{g^{c}},\D^{c})$ is a regular nonholonomic system.
		\item[(ii)] Let $\Gamma_{(L_{g^{c}},\D^{c})}\in \mathfrak{X}(\D^{c})$ be the nonholonomic SODE associated with the system $(L_{g^{c}},\D^{c})$ and $\kappa_{Q}:TTQ\rightarrow TTQ$ the canonical involution. Then
		\begin{equation}
			\Gamma_{(L_{g^{c}},\D^{c})}=T\kappa_{Q}|_{T\D} \circ \Gamma_{(L_{g},\D)}^{c} \circ \kappa_{Q}|_{\D^{c}}
		\end{equation}
		and so we have
		\begin{enumerate}
			\item[(a)] $\Gamma_{(L_{g^{c}},\D^{c})}$ is $T\tau_{Q}|_{\D^{c}}$-projectable over $\Gamma_{(L_{g},\D)}$; \\
			\item[(b)] The trajectories of $\Gamma_{(L_{g^{c}},\D^{c})}$ are just the Jacobi fields for the nonholonomic system $(L_{g},\D)$.
		\end{enumerate}		
	\end{enumerate}
\end{theorem}

If $c_{v}:I\rightarrow Q$ is a trajectory of nonholonomic dynamics and $W:I\rightarrow TQ$ is a vector field on $Q$ along $c_{v}$ then an immediate corollary of this theorem is that

\begin{corollary}
	$W$ is a Jacobi field for the nonholonomic system $(L_{g},\D)$ if and only if
	\begin{enumerate}
		\item $\dot{W}(t)\in \D^{c}_{W(t)}$, for every $t\in I$;
		\item $i_{\ddot{W}}\omega_{L_{g^{c}}}(\dot{W})-dL_{g^{c}}(\dot{W})\in (S^{T})^{*}((T\D^{c})^{o})$,
	\end{enumerate}
	where {$\omega_{L_{g^{c}}}$ is the Poincar\'e-Cartan $2$-form associated with the Lagrangian function $L_{g^{c}}$ and  $S^{T}:T(TTQ)\rightarrow T(TTQ)$} is the vertical endomorphism on $TTQ$.
\end{corollary}

First we show that the complete lift $(L_{g^{c}},\D^{c})$ on $TQ$ obtained from the nonholonomic system of kinetic type $(L_{g},\D)$ on $Q$ is always regular.

\begin{proposition}\label{regularityprop}
	If $L_{g}$ is the Lagrangian function associated to the Riemannian metric $g$, then the nonholonomic system $(L_{g^{c}},\D^{c})$ is regular.
\end{proposition}

\begin{proof}
	Let $Z\in\D^{c} \cap (\D^{c})^{\bot}$. Let $\{X^{a}\}$ be an orthonormal basis of sections on $\D$. The set $\{(X^{a})^{v},(X^{a})^{c}\}$ is then a basis of sections of $\D^{c}$ and $Z$ may be written as
	\begin{equation*}
		Z=\lambda_{a}(X^{a})^{c}+\mu_{a}(X^{a})^{v}.
	\end{equation*}
	Since $Z$ is in the intersection of $\D^{c}$ with its $g^{c}-$orthogonal distribution then, using \eqref{def:complete:lift:g} in Appendix C, we have that for every $Y\in\Gamma(D^{c})$ expressed as $Y=f_{a}(X^{a})^{c}+g_{a}(X^{a})^{v}$ in the same basis, 
	\begin{equation*}
		\begin{split}
			0 & =g^{c}(Z,Y) \\
			& = \lambda_{a}f_{b}(g(X^{a},X^{b}))^{c}+(\lambda_{a}g_{b}+\mu_{a}f_{b})(g(X^{a},X^{b}))^{v} \\
			& = \lambda_{a}g_{a}+\mu_{a}f_{a},
		\end{split}
	\end{equation*}
	since we are taking an orthonormal basis of $\D$. Since the functions $f_{a}$ and $g_{a}$ are arbitrary, we deduce that $\lambda_{a}=\mu_{a}=0$, hence, $Z=0$. Therefore, by Theorem \ref{regularityth} the nonholonomic system $(L_{g^{c}},\D^{c})$ is regular.
\end{proof}

The last proposition proves item (i) in Theorem \ref{MainTheorem}. Therefore, from now on we can refer to the nonholonomic SODE $\Gamma_{(L_{g^{c}},\D^{c})}$ associated with the complete lift system $(L_{g^{c}},\D^{c})$.

In order to prove item (ii) in Theorem \ref{MainTheorem} we will characterize further the distribution $\D^{c}$. Our main purpose is to identify a local basis of the distribution $(S^{T})^{*}((T\D^{c})^{o})$ along $\D^{c}$, where $S^{T}:T(TTQ)\rightarrow T(TTQ)$ is the vertical endomorphism on $TTQ$.

If $\mu$ is a 1-form on $Q$, we will denote by $\mu^{c}\in\Omega^{1}(TQ)$ and $\mu^{v}\in\Omega^{1}(TQ)$ the complete and vertical lifts, respectively, of $\mu$ to $TQ$ (see \eqref{1formcomplete} and \eqref{1formvertical}).

\begin{lemma} \label{Flemma}
	Let $\tilde{F}$ be the distribution along $\D^{c}$ defined by $\tilde{F}^{o}=(S^{T})^{*}((T\D^{c})^{o})$. 
	\begin{enumerate}
		\item Given $\D^{o}$, define a distribution $F$ along $\D$ by $F^{o}=S^{*}(T\D^{o})$. Then 
		\begin{equation*}
			\Gamma((\D^{c})^{o})=\langle \{ \mu^{c},\mu^{v} | \mu\in\Gamma(\D^{o})\} \rangle \quad \text{and} \quad F^{o}=\langle \{\mu^{v} | \mu\in\Gamma(\D^{o})\} \rangle.
		\end{equation*}
		\item $\tilde{F}^{o}=(S^{T})^{*}((T\D^{c})^{o})$ is spanned by $\tilde{F}^{o}=\langle \{ (\mu^{c})^{v},(\mu^{v})^{v} \} \rangle$. Moreover,
		\begin{equation*}
			\kappa_{Q}^{*}\tilde{F}^{o}=\langle \{ (\mu^{v})^{c},(\mu^{v})^{v}| \mu\in\Gamma(\D^{o}) \} \rangle.
		\end{equation*}
	\end{enumerate}
\end{lemma}

\begin{proof}
	\begin{enumerate}
		\item For every $X\in\Gamma(\D)$ and $\mu\in\Gamma(\D^{o})$ we have that $\langle \mu,X \rangle=0$. Moreover, note that we have the following identities
		$$\langle \mu^{c},X^{c} \rangle=(\langle \mu,X \rangle)^{c}=0,$$
		$$\langle \mu^{c},X^{v} \rangle=(\langle \mu,X \rangle)^{v}=0,$$
		$$\langle \mu^{v},X^{c} \rangle=(\langle \mu,X \rangle)^{v}=0,$$
		$$\langle \mu^{v},X^{v} \rangle=0,$$	
		hence the elements in $\{ \mu^{c},\mu^{v} \}$ annihilate $\D^{c}$. Therefore, by dimensional reasons they must span $(\D^{c})^{o}$.
		
		It is also easy to show that $(T\D)^{o}$ is spanned by the 1-forms $d\widehat{\mu}$, where $\widehat{\mu}$ is the fiberwise linear function associated to $\mu$. Hence, using the fact that $S^{*}(d\widehat{\mu})=\mu^{v}$ (see \eqref{S*lifts} in Appendix A), it follows that 1-forms of the form $\mu^{v}$ span $F^{o}$.
		\item On one hand, to see that $\tilde{F}^{o}$ is generated by the elements of the form $(\mu^{c})^{v}$ and $(\mu^{v})^{v}$ is a direct application of the previous item. Of course, we have that
		\begin{equation*}
			(T\D^{c})^{o}=\langle \{d\widehat{(\mu^{c})}, d\widehat{(\mu^{v})} | \mu\in\Gamma(\D^{o})\} \rangle
		\end{equation*}
		and then we may use \eqref{S*lifts} in Appendix A. Here, $\widehat{(\mu^{c})}$ and $\widehat{(\mu^{v})}$ are the fiberwise linear functions on $TTQ$ induced by the 1-forms $\mu^{c}$ and $\mu^{v}$. Hence, we obtain
		\begin{equation*}
			\tilde{F}^{o}=\langle \{ (\mu^{c})^{v},(\mu^{v})^{v} \} \rangle.
		\end{equation*}
		The last part of the Lemma follows using \eqref{codiff:can:invol}.
		
	\end{enumerate}
\end{proof}

\begin{proof}[Proof of Theorem \ref{MainTheorem}]
	By Proposition \ref{regularityprop}, the complete lift nonholonomic system $(L_{g^{c}},\D^{c})$ is regular. The nonholonomic vector field $\Gamma_{(L_{g},\D)}$ is defined by the equations
		\begin{equation*}
			\left\{\begin{array}{l}
				\left. \left( i_{\Gamma_{(L_{g},\D)}}\omega_{L_{g}}-dE_{L_{g}} \right) \right|_{\D} \in \Gamma(S^*((T\D)^{o})) \\
				\Gamma_{(L_{g},\D)} \in {\frak X}(\D).
			\end{array}\right.
		\end{equation*}
	Using the complete lift and \eqref{differential:completelift} and \eqref{inner:complete} in Appendix A, we can obtain the following equation
		\begin{equation*}
			\left. \left( i_{\Gamma_{(L_{g},\D)}^{c}}\omega_{L_{g}}^{c}-dE_{L_{g}}^{c} \right) \right|_{T\D} \in \Gamma((S^*((T\D)^{o}))^{c})
		\end{equation*}
	If we pullback the previous equation by $\kappa_{Q}$ and using the characterization of $\tilde{F}$ given in Lemma \ref{Flemma} we deduce that
	\begin{equation*}
		\left. \left( i_{(\kappa_{Q})_{*}\Gamma_{(L_{g},\D)}^{c}}\kappa_{Q}^{*}\omega_{L_{g}}^{c}-d(\kappa_{Q}^{*}E_{L_{g}}^{c}) \right) \right|_{\D^{c}} \in \Gamma(\tilde{F}^{o})
	\end{equation*}
	Applying Lemma \ref{liftedforms} in Appendix C, the equation reduces to 
	\begin{equation*}
		\left. \left( i_{(\kappa_{Q})_{*}\Gamma_{(L_{g},\D)}^{c}}\omega_{L_{g^{c}}}-d(E_{L_{g^{c}}}) \right) \right|_{\D^{c}} \in \Gamma(F^{o})
	\end{equation*}
	Notice that since $\Gamma_{(L_{g},\D)}$ is a vector field in the submanifold  $\D$, its complete lift satisfies $\Gamma_{(L_{g},\D)}^{c}\in\mathfrak{X}(T\D)$. Therefore we may form the commutative diagram below
	\begin{figure}[!htb]
		\begin{center}
			\begin{tikzcd}
				T\mathcal{D}^{c} \arrow[rr, "T\kappa_{Q}|_{T\D}"] \arrow[dd, "\tau_{TTQ}"]                                                                       && TT\mathcal{D} \arrow[dd, "\tau_{TTQ}"']                                 \\
				\\
				\mathcal{D}^{c} \arrow[rr, "\kappa_{Q}|_{\mathcal{D}^{c}}"] \arrow[uu, "{(\kappa_{Q}^{-1})_{*}\Gamma_{(L_g,\mathcal{D})}^{c}}", bend left=49] && T\mathcal{D} \arrow[uu, "{\Gamma_{(L_g,\mathcal{D})}^{c}}"', bend right=49]
			\end{tikzcd}
		\end{center}
	\end{figure}

	Hence, $(\kappa_{Q})_{*}\Gamma_{(L_{g},\D)}^{c}$ is a vector field on $\D^{c}$. Moreover, since the nonholonomic system $(L_{g^{c}},\D^{c})$ is regular, by uniqueness of nonholonomic vector field, it coincides with $\Gamma_{(L_{g^{c}},\D^{c})}$, i.e.,
	\begin{equation}
	\Gamma_{(L_{g^{c}},\D^{c})}=T\kappa_{Q}|_{T\D} \circ \Gamma_{(L_{g},\D)}^{c} \circ \kappa_{Q}|_{\D^{c}}.
	\end{equation}

	Then the statements in item $(ii)$ are just consequences of the properties of the complete lift and the canonical involution. Indeed,
	\begin{equation*}
		\begin{split}
			TT\tau_{Q}(\Gamma_{(L_{g^{c}},\D^{c})}) & =T(T\tau_{Q}\circ\kappa_{Q}|_{T\D}) \circ \Gamma_{(L_{g},\D)}^{c} \circ \kappa_{Q}|_{\D^{c}} \\
			& = T(\tau_{TQ}|_{T\D})(\Gamma_{(L_{g},\D)}^{c} \circ \kappa_{Q}|_{\D^{c}}) \\
			& = \Gamma_{(L_{g},\D)}\circ \tau_{TQ}|_{T\D} \circ\kappa_{Q}|_{\D^{c}} \\
			& = \Gamma_{(L_{g},\D)},
		\end{split}
	\end{equation*}
	where we have used that $\tau_{TQ} \circ\kappa_{Q}(\D^{c})=\D$. This proves the first statement.
	
	The second statement in item $(ii)$, may be seen from the fact that if $W:I\rightarrow TQ$ is a trajectory of $\Gamma_{(L_{g^{c}},\D^{c})}$, then its tangent lift $\dot{W}:I\rightarrow \D^{c}$ is an integral curve of $\Gamma_{(L_{g^{c}},\D^{c})}$ and, thus, $\kappa_{Q}\circ\dot{W}:I\rightarrow T\D$ is an integral curve of $\Gamma_{(L_{g},\D)}^{c}$. Therefore we may write it as
	\begin{equation*}
		\kappa_{Q}\circ\dot{W}(t)=\left( T_{W(0)} \phi_{t}^{\Gamma_{(L_{g},\D)}} \right) (\kappa_{Q}\circ\dot{W}(0)).
	\end{equation*}
	So,
	\begin{equation*}
		W(t)=T\tau_{Q}(\kappa_{Q}\circ\dot{W})=T\tau_{Q}\left( \left( T_{W(0)} \phi_{t}^{\Gamma_{(L_{g},\D)}} \right) (\kappa_{Q}\circ\dot{W}(0)) \right)
	\end{equation*}
	and
	\begin{equation*}
		W(t)= \left( T_{W(0)} (\tau_{Q}\circ \phi_{t}^{\Gamma_{(L_{g},\D)}}) \right) (\kappa_{Q}\circ\dot{W}(0)).
	\end{equation*}
	Let now $v:I\rightarrow \D$ be a curve such that its initial velocity is $v'(0)=\kappa_{Q}\circ\dot{W}(0)$. Then
	\begin{equation*}
	W(t)= \left. \frac{d}{ds} \right|_{s=0} \left( \tau_{Q}\circ \phi_{t}^{\Gamma_{(L_{g},\D)}}\right) (v(s)).
	\end{equation*}
	Hence, $W$ is a nonholonomic Jacobi field for $\Gamma_{(L_{g},\D)}$, since it is an infinitesimal variation of nonholonomic trajectories of $\Gamma_{(L_{g},\D)}$.	
\end{proof}

\begin{remark}
	As a consequence of the last theorem if $W:I\longrightarrow TQ$ is a Jacobi field for the nonholonomic dynamics $(L_{g},\D)$ it must satisfy the constraint:
	\begin{equation*}
	\dot{W}\in D^{c}.
	\end{equation*}
\end{remark}

\begin{example}
	Let us check that the lifted nonholonomic system obtained from the nonholonomic particle is regular.
	
	By Theorem \ref{regularityth} it is enough to check that $\D^{c}\cap (\D^{c})^{\bot}=\{0\}$. This is equivalent to show that the matrix $C^{a b}$ defined in \eqref{matrixC} is non-singular at points of $\D^{c}$. If we were to compute this matrix we would find it was
	\begin{equation*}
		\begin{pmatrix}
			0 & y^{2}+1 \\
			y^{2}+1 & 2vy
		\end{pmatrix}
	\end{equation*}
	which is clearly non-singular.
	
	In this example the constraint distribution $\D$ is generated by the vectors $e_{1}=\frac{\partial}{\partial x}+y\frac{\partial}{\partial z}$ and $e_{2}=\frac{\partial}{\partial y}$. The orthogonal distribution $\D^{\bot}$ for the euclidean metric is generated by $e_{3}=y\frac{\partial}{\partial x}-\frac{\partial}{\partial z}$.
	
	The lifted distribution $\D^{c}$, by definition, is generated by the vectors $e_{1}^{c},e_{2}^{c},e_{1}^{v},e_{2}^{v}$. The set $\{e_{3}^{c},e_{3}^{v}\}$ is linearly independent and it is easily proven to be $g^{c}-$orthogonal to $\D^{c}$, hence, by dimensional reasons, it generates the orthogonal distribution $(\D^{c})^{\bot}$.
	
	Moreover, since $\{e_{1}^{c},e_{2}^{c},e_{1}^{v},e_{2}^{v},e_{3}^{c},e_{3}^{v}\}$ is a basis of sections of $\mathfrak{X}(TQ)$, the intersection of $\D^{c}$ and $(\D^{c})^{\bot}$ must be zero.	
\end{example}


\subsection{Nonholonomic Jacobi equation}

Theorem \ref{nhgeodesicth} asserts that if $c_{v}:I\rightarrow Q$ is a trajectory of $\Gamma_{(L_{g}, \D)}$, then 
\begin{equation*}
\nabla^{nh}_{\dot{c}_{v}}\dot{c}_{v}=0 \quad \text{and} \quad \dot{c}_{v}\in\D.
\end{equation*}

Consider the complete lift of the metric $g$ denoted by $g^{c}$, which is a symmetric non-degenerate $(0,2)$-tensor on $TQ$. The kinetic Lagrangian $L_{g^{c}}$ associated to $g^{c}$ satisfies $L_{g}^{c}\circ \kappa_{Q}=L_{g^{c}}$ (see Lemma \ref{liftedLagrangianlemma} in Appendix C). Moreover, $(L_{g^{c}},\D^{c})$ is a regular nonholonomic system.

Since the Lagrangian function is kinetic, its trajectories are geodesics of the connection $\nabla^{NH}$ defined by
\begin{equation}
\nabla^{NH}_{X} Y:=P^{T}(\nabla_{X}^{g^{c}} Y)+\nabla^{g^{c}}_{X}[P'^{T}(Y)], \quad \text{for} \ X,Y \in \mathfrak{X}(TQ),
\end{equation}
where $\nabla^{g^{c}}$ is the Levi-Civita connection of $g^{c}$, $P^{T}:TTQ\rightarrow \D^{c}$ is the associated  orthogonal projector onto the distribution $\D^{c}$ and $P'^{T}:TTQ\rightarrow (\D^{c})^{\bot}$ is the orthogonal projector onto $(\D^{c})^{\bot}$, the orthogonal distribution.

\begin{lemma}\label{Projectorlift}
The following identities are satisfied:
	\begin{enumerate}
		\item $\nabla^{g^{c}}=(\nabla^{g})^{c}$;
		\item $\kappa_{Q}\circ TP\circ \kappa_{Q}(X^{c})=(P(X))^{c}$, for any $X\in \mathfrak{X}(Q)$;
		\item $\kappa_{Q}\circ TP\circ \kappa_{Q}(X^{v})=(P(X))^{v}$, for any $X\in \mathfrak{X}(Q)$;
		\item $P^{T}=\kappa_{Q}\circ TP\circ \kappa_{Q}$;
		\item $P'^{T}=\kappa_{Q}\circ TP'\circ \kappa_{Q}$.
	\end{enumerate}
\end{lemma}

\begin{proof}
	The first item is proved in Corollary 2.6.6. in \citep*{LR89}. To prove item 2, just use the properties of the canonical involution in Appendix C (see \eqref{int:def:can:invol} in Appendix C)
	\begin{equation*}
		\kappa_{Q}\circ TP\circ \kappa_{Q}(X^{c})=\kappa_{Q}\circ TP (TX)=\kappa_{Q}(T(P\circ X))=(P(X))^{c}.
	\end{equation*}
	We may prove item 3 in a similar way. Given $u_{q}\in T_{q}Q$, we have
	\begin{equation*}
		\begin{split}
			\kappa_{Q}\circ TP\circ \kappa_{Q}(X^{v})(u_{q}) & =(\kappa_{Q}\circ TP)(T_{q}0(u_{q})+(X(q))^{v}|_{0_{q}}) \\
			 & =(\kappa_{Q}\circ TP)(T_{q}0(u_{q})+\left.\frac{d}{dt}\right|_{t=0}(tX(q))) \\
			 & = \kappa_{Q}(T_{q}0(u_{q})+\left.\frac{d}{dt}\right|_{t=0}(tPX(q))) \\
			 & = (PX)^{v}(u_{q}).
		\end{split}
	\end{equation*}
	As a consequence of the two previous items we have that
	\begin{equation*}
	\begin{split}
		&\kappa_{Q}\circ TP\circ \kappa_{Q}(X^{c})=X^{c}, \quad \kappa_{Q}\circ TP\circ \kappa_{Q}(X^{v})=X^{v}, \quad X\in \Gamma(\D) \\
		&\kappa_{Q}\circ TP\circ \kappa_{Q}(Y^{c})=0, \quad \kappa_{Q}\circ TP\circ \kappa_{Q}(Y^{v})=0, \quad Y\in \Gamma(\D^{\bot}).
	\end{split}
	\end{equation*}
	Note that while $\{X^{c},X^{v}| X\in \Gamma(\D)\}$ spans $\Gamma(\D^{c})$, the set $\{Y^{c},Y^{v}| Y\in \Gamma(\D^{\bot})\}$ spans $\Gamma((\D^{c})^{\bot})$, where the orthogonal is taken with respect to the pseudo-Riemannian metric $g^{c}$. Hence, $\kappa_{Q}\circ TP\circ \kappa_{Q}$ is the identity on $\D^{c}$ and vanishes on $(\D^{c})^{\bot}$. Therefore, it must be the orthogonal projector $P^{T}$.
	
	The argument to prove item 5. is completely analogous, just substitute $P$ by $P'$.
\end{proof}

The last Lemma simplifies the proof of the next Proposition, relating both nonholonomic connections by the complete lift. Before, the statement let us recall some properties of the complete lift of a linear connection $\nabla$ (see \citep*{LR89} or \citep*{YaIs73}):
\begin{equation}\label{connectionlift}
	\nabla^{c}_{X^{c}}Y^{c}=(\nabla_{X}Y)^{c}, \quad \nabla^{c}_{X^{c}}Y^{v}=\nabla^{c}_{X^{v}}Y^{c}=(\nabla_{X}Y)^{v}, \quad \nabla^{c}_{X^{v}}Y^{v}=0,
\end{equation}
for any $X, Y \in \mathfrak{X}(Q)$.

\begin{proposition}\label{completeconnection}
The nonholonomic connection constructed from the Levi-Civita connection associated to $g^{c}$ and from the projectors $P^{T}$, $P'^{T}$ is the complete lift of the nonholonomic connection constructed from the Levi-Civita for $g$ and from the projector $P$, and $P'$. In other words,
\begin{equation*}
	\nabla^{NH}=(\nabla^{nh})^{c}.
\end{equation*}
\end{proposition}

\begin{proof}
	We will prove the identity on complete and vertical lifts. Using the definition of $\nabla^{NH}$ we get
	\begin{equation*}
		\nabla^{NH}_{X^{c}}Y^{c} =P^{T}(\nabla^{g})^{c}_{X^{c}} Y^{c}+(\nabla^{g})^{c}_{X^{c}}[P'^{T}(Y^{c})].
	\end{equation*}
	Using the properties stated in equations \eqref{connectionlift} and in Lemma \ref{Projectorlift} we deduce
	\begin{equation*}
		\nabla^{NH}_{X^{c}}Y^{c} = P^{T}(\nabla^{g}_{X}Y)^{c}+(\nabla^{g})^{c}_{X^{c}}(P'Y)^{c}.
	\end{equation*}
	Again applying the same combination of arguments we may reduce the previous line to
	\begin{equation*}
		\nabla^{NH}_{X^{c}}Y^{c} = (P\nabla^{g}_{X}Y)^{c}+(\nabla^{g}_{X}P'Y)^{c},
	\end{equation*}
	which is just the complete lift of $\nabla^{nh}$. So,
	\begin{equation*}
		\nabla^{NH}_{X^{c}}Y^{c} = (\nabla^{nh}_{X}Y)^{c}=(\nabla^{nh})^{c}_{X^{c}}Y^{c}.
	\end{equation*}
	The very same arguments are still valid to prove
	\begin{equation*}
		\begin{split}
			\nabla^{NH}_{X^{c}}Y^{v} & =P^{T}(\nabla^{g})^{c}_{X^{c}} Y^{v}+(\nabla^{g})^{c}_{X^{c}}[P'^{T}(Y^{v})] \\
			& = P^{T}(\nabla^{g}_{X}Y)^{v}+(\nabla^{g})^{c}_{X^{c}}(P'Y)^{v} \\
			& = (P\nabla^{g}_{X}Y)^{v}+(\nabla^{g}_{X}P'Y)^{v} \\
			& = (\nabla^{nh}_{X}Y)^{v}=(\nabla^{nh})^{c}_{X^{c}}Y^{v},
		\end{split}
	\end{equation*}
	and also to prove
	\begin{equation*}
	\begin{split}
	\nabla^{NH}_{X^{v}}Y^{v} & =P^{T}(\nabla^{g})^{c}_{X^{v}} Y^{v}+(\nabla^{g})^{c}_{X^{v}}[P'^{T}(Y^{v})] \\
	& = (\nabla^{g})^{c}_{X^{v}}(P'Y)^{v}= 0=(\nabla^{nh})^{c}_{X^{v}}Y^{v}.
	\end{split}
	\end{equation*}
	
\end{proof}

\begin{remark}
	If $W:I\rightarrow TQ$ is a trajectory of the nonholonomic system $(L_{g^{c}},\D^{c})$, it is also by Theorem \ref{MainTheorem} a Jacobi field for the nonholonomic system $(L_{g},\D)$, and it is a geodesic for the nonholonomic connection $\nabla^{NH}$ by Theorem \ref{nhgeodesicth}. Hence, by the last proposition $W$ satisfies
	\begin{equation*}
	(\nabla^{nh})^{c}_{\dot{W}}\dot{W}=0.
	\end{equation*}
\end{remark}

\begin{proposition} \label{complete:connection:coordinates}
	Let $W:I\rightarrow TQ$ be a vector field along $c:I\rightarrow Q$, a nonholonomic trajectory of $\Gamma_{(L_{g}, \D)}$. Then the coordinate expression of $(\nabla^{nh})^{c}_{\dot{W}}\dot{W}$ is
	\begin{equation}\label{liftedgeo}
		(\nabla^{nh})^{c}_{\dot{W}}\dot{W}=\left(\frac{d^{2} W^{k}}{dt^{2}}+\dot{q}^{i}\dot{q}^{j}W^{l}\frac{\partial \Gamma_{ij}^{k}}{\partial q^{l}}+\dot{q}^{j}\frac{d W^{i}}{dt}(\Gamma_{ij}^{k}+\Gamma_{ji}^{k})\right)\frac{\partial}{\partial \dot{q}^{k}},
	\end{equation}
	where $(q^{i})$ are local coordinates on $Q$ with respect to which the local expression of $W$ is
	\begin{equation*}
		W(t)=W^{i}(t)\left. \frac{\partial}{\partial q^{i}}\right|_{c(t)},
	\end{equation*}
	$(q^{i},\dot{q}^{i})$ is the corresponding local expression of $\dot{c}$ on $TQ$ and $\Gamma_{ij}^{k}$ are the Chrystoffel symbols for the nonholonomic connection $\nabla^{nh}$, i.e.,
	\begin{equation*}
		\nabla^{nh}_{\frac{\partial}{\partial q^{i}}} \frac{\partial}{\partial q^{j}}=\Gamma_{ij}^{k}\frac{\partial}{\partial q^{k}}.
	\end{equation*}
\end{proposition}

\begin{proof}
Denote by $\dot{W}:I\rightarrow TTQ$ the tangent lift of $W:I\rightarrow TQ$. Then we have that
	\begin{equation*}
		\dot{W}(t)=\dot{q}^{i}(t)\left.\frac{\partial}{\partial q^{i}}\right|_{W(t)}+\dot{W}^{i}(t)\left.\frac{\partial}{\partial \dot{q}^{i}}\right|_{W(t)}.
	\end{equation*}
Observe that the coordinate vector fields on $TQ$, denoted by $\frac{\partial}{\partial q^{i}}$ and $\frac{\partial}{\partial \dot{q}^{i}}$ are just the complete and the vertical lift of the corresponding vector field on $Q$, i.e.,
	\begin{equation*}
		\frac{\partial}{\partial q^{i}}(v_{q})=\left(\frac{\partial}{\partial q^{i}}\right)^{c}(v_{q}), \quad \frac{\partial}{\partial \dot{q}^{i}}(v_{q})=\left(\frac{\partial}{\partial q^{i}}\right)^{v}(v_{q}).
	\end{equation*}
With these properties in mind it is easy to prove that,
	\begin{equation*}
		\begin{split}
			& (\nabla^{nh})^{c}_{\frac{\partial}{\partial q^{i}}}\frac{\partial}{\partial q^{j}} = \Gamma_{ij}^{k}\frac{\partial}{\partial q^{k}}+\dot{q}^{l}\frac{\partial \Gamma_{ij}^{k}}{\partial q^{l}}\frac{\partial}{\partial \dot{q}^{k}} \\
			& (\nabla^{nh})^{c}_{\frac{\partial}{\partial \dot{q}^{i}}}\frac{\partial}{\partial q^{j} } = (\nabla^{nh})^{c}_{\frac{\partial}{\partial q^{i}}}\frac{\partial}{\partial \dot{q}^{j}} = \Gamma_{ij}^{k}\frac{\partial}{\partial \dot{q}^{k}}\\
			& (\nabla^{nh})^{c}_{\frac{\partial}{\partial \dot{q}^{i}}}\frac{\partial}{\partial \dot{q}^{j}} = 0.
		\end{split}
	\end{equation*} 

And without further ado, one can also compute $(\nabla^{nh})^{c}_{\dot{W}}\dot{W}$ to be
	\begin{equation}\label{liftedgeo2}
	\begin{split}
		(\nabla^{nh})^{c}_{\dot{W}}\dot{W}=&\left(\ddot{q}^{k}+\Gamma_{ij}^{k}\dot{q}^{i}\dot{q}^{j}\right)\frac{\partial}{\partial q^{k}}\\
		&+\left(\ddot{W}^{k}+\dot{q}^{i}\dot{q}^{j}W^{l}\frac{\partial \Gamma_{ij}^{k}}{\partial q^{l}}+\dot{q}^{j}\dot{W}^{i}(\Gamma_{ij}^{k}+\Gamma_{ji}^{k})\right)\frac{\partial}{\partial \dot{q}^{k}}\; .
		\end{split}
	\end{equation}
The first term vanishes since $c$ is a geodesic for $\nabla^{nh}$ by Theorem \ref{nhgeodesicth}. Hence, we get the expected result. 
\end{proof}

Denote by $T^{nh}$ and $R^{nh}$ the torsion and curvature tensors, respectively, associated with the nonholonomic connection $\nabla^{nh}$, that is,
\begin{equation*}
	\begin{split}
		T^{nh}(X,Y) & =\nabla^{nh}_{X} Y-\nabla^{nh}_{Y} X - [X,Y], \\
		R^{nh}(X,Y)Z & = \nabla^{nh}_{X}\nabla^{nh}_{Y}Z-\nabla^{nh}_{Y}\nabla^{nh}_{X}Z-\nabla^{nh}_{[X,Y]}Z,
	\end{split}
\end{equation*}
for $X, Y, Z \in \mathfrak{X}(Q)$. Then, using $T^{nh}$ and $R^{nh}$, we will obtain a characterization of the nonholonomic Jacobi fields with an equation which may be considered as the version for kinematic nonholonomic systems of the Jacobi equation in Riemannian geometry.

\begin{theorem}\label{Jacobi:equation}
Let $(L_{g},\D)$ be a kinematic nonholonomic system, $\nabla^{nh}$ the nonholonomic connection on $Q$ with torsion and curvature tensors denoted by $T^{nh}$ and $R^{nh}$, respectively, and $W:I\rightarrow TQ$ a vector field along a nonholonomic trajectory $c:I\rightarrow Q$. Then $W$ is a nonholonomic Jacobi field if and only if 
	\begin{equation}\label{nhjacobi}
		\nabla^{nh}_{\dot{c}}\nabla^{nh}_{\dot{c}}W+\nabla^{nh}_{\dot{c}}T^{nh}(W,\dot{c})+R^{nh}(W,\dot{c})\dot{c}=0, \quad \dot{W}(t)\in \D^{c}.
	\end{equation}
\end{theorem}

\begin{proof}
	Using the same notation introduced both in the statement of the last proposition as well as in its proof, let us compute the coordinate expression of the left-hand side of equation \eqref{nhjacobi}.
	
	It is easy to see that
	\begin{equation*}
		\nabla^{nh}_{\dot{c}}W= \left(\dot{W}^{k}+\dot{q}^{i}W^{j}\Gamma_{ij}^{k}\right)\frac{\partial}{\partial q^{k}}.
	\end{equation*}
	Computing the second covariant derivative we obtain
	\begin{equation}\label{secondcovariant}
		\begin{split}
			\nabla^{nh}_{\dot{c}}\nabla^{nh}_{\dot{c}}W= \left( \ddot{W}^{m}+2\dot{W}^{j}\dot{q}^{i}\Gamma_{ij}^{m}+\ddot{q}^{i}W^{j}\Gamma_{ij}^{m}+\dot{q}^{i}W^{j}\Gamma_{ij}^{k}\dot{q}^{l}\Gamma_{lk}^{m}  \right. & \\
			\left. +\dot{q}^{i}W^{j}\frac{\partial \Gamma_{ij}^{m}}{\partial q^{l}}\dot{q}^{l} \right) & \frac{\partial}{\partial q^{m}}.
		\end{split}
	\end{equation}
	Now the term with the curvature tensor appearing in equation \eqref{nhjacobi} is
	\begin{equation}\label{curvature}
		R^{nh}(W,\dot{c})\dot{c}=W^{i}\dot{q}^{j}\dot{q}^{l}\left( \frac{\partial \Gamma_{jl}^{m}}{\partial q^{i}}+\Gamma_{jl}^{k}\Gamma_{ik}^{m}-\frac{\partial \Gamma_{il}^{m}}{\partial q^{j}}-\Gamma_{il}^{k}\Gamma_{jk}^{m} \right)\frac{\partial}{\partial q^{m}},
	\end{equation}
	while the torsion tensor is
	\begin{equation*}
		T^{nh}(W,\dot{c})=W^{i}\dot{q}^{j}T_{ij}^{m}\frac{\partial}{\partial q^{m}}, \quad \text{with} \ T_{ij}^{m}=\Gamma_{ij}^{m}-\Gamma_{ji}^{m},
	\end{equation*}
	and the term involving the torsion tensor is
	\begin{equation}\label{torsion}
		\nabla^{nh}_{\dot{c}}T^{nh}(W,\dot{c})=\left( \dot{W}^{i}\dot{q}^{j}T_{ij}^{m}+W^{i}\ddot{q}^{j}T_{ij}^{m}+W^{i}\dot{q}^{j}\frac{\partial T_{ij}^{m}}{\partial q^{l}}\dot{q}^{l}+W^{i}\dot{q}^{j}T_{ij}^{k}\dot{q}^{l}\Gamma_{lk}^{m} \right)\frac{\partial}{\partial q^{m}}
	\end{equation}
	Now we will add the three terms appearing in equation \eqref{nhjacobi} to obtain that the sum is equal to	
	\begin{equation}\label{goal}
	\left( \ddot{W}^{m}+\dot{q}^{i}\dot{q}^{j}W^{l}\frac{\partial \Gamma^{m}_{ij}}{\partial q^{l}}+2\dot{q}^{i}\dot{W}^{j}\Gamma^{m}_{ij}+\dot{W}^{i}\dot{q}^{j}T_{ij}^{m} \right)\frac{\partial}{\partial q^{m}},
	\end{equation}
	which implies that
	\begin{equation*}
	(\nabla^{nh})^{c}_{\dot{W}}\dot{W}=\left( \nabla^{nh}_{\dot{c}}\nabla^{nh}_{\dot{c}}W+\nabla^{nh}_{\dot{c}}T^{nh}(W,\dot{c})+R^{nh}(W,\dot{c})\dot{c} \right)^{v}.
	\end{equation*}
	Indeed note that adding the third term in \eqref{secondcovariant} with the second one in \eqref{curvature} we get
	\begin{equation*}
		\ddot{q}^{i}W^{j}\Gamma_{ij}^{m}+W^{i}\dot{q}^{j}\dot{q}^{l}\Gamma_{jl}^{k}\Gamma_{ik}^{m}=W^{i}\dot{q}^{j}\dot{q}^{l}\Gamma_{jl}^{k}T_{ik}^{m},
	\end{equation*}
	adding the fourth term in \eqref{secondcovariant} with the last term in \eqref{curvature}
	\begin{equation*}
		\dot{q}^{i}\dot{q}^{l}W^{j}\Gamma_{ij}^{k}\Gamma_{lk}^{m}-W^{i}\dot{q}^{j}\dot{q}^{l}\Gamma_{il}^{k}\Gamma_{jk}^{m}=W^{j}\dot{q}^{i}\dot{q}^{l}T_{ij}^{k}\Gamma_{lk}^{m}
	\end{equation*}
	and adding the last term in \eqref{secondcovariant} with the first and third terms in \eqref{curvature} we get
	\begin{equation*}
		\dot{q}^{i}\dot{q}^{l}W^{j}\frac{\partial \Gamma_{ij}^{m}}{\partial q^{l}}+W^{i}\dot{q}^{j}\dot{q}^{l}\left( \frac{\partial \Gamma_{jl}^{m}}{\partial q^{i}}-\frac{\partial \Gamma_{il}^{m}}{\partial q^{j}} \right)=\dot{q}^{i}\dot{q}^{l}W^{j}\left( \frac{\partial T_{ij}^{m}}{\partial q^{l}}+\frac{\partial \Gamma_{il}^{m}}{\partial q^{j}} \right).
	\end{equation*}
	The sum of $\nabla^{nh}_{\dot{c}}\nabla^{nh}_{\dot{c}}W$ and $R^{nh}(W,\dot{c})\dot{c}$ is
	\begin{equation*}
		\left[ \ddot{W}^{m}+2\dot{W}^{j}\dot{q}^{i}\Gamma_{ij}^{m}+W^{i}\dot{q}^{j}\dot{q}^{l}\Gamma_{jl}^{k}T_{ik}^{m}+W^{j}\dot{q}^{i}\dot{q}^{l}T_{ij}^{k}\Gamma_{lk}^{m}+\dot{q}^{i}\dot{q}^{l}W^{j}\left( \frac{\partial T_{ij}^{m}}{\partial q^{l}}+\frac{\partial \Gamma_{il}^{m}}{\partial q^{j}} \right) \right]\frac{\partial}{\partial q^{m}}.
	\end{equation*}
	Comparing the expression above with our goal, which is to prove that the sum of the three terms is equal to \eqref{goal}, the result would be proven if we establish that
	\begin{equation*}
		\left( \dot{W}^{i}\dot{q}^{j}T_{ij}^{m}-W^{i}\dot{q}^{j}\dot{q}^{l}\Gamma_{jl}^{k}T_{ik}^{m}-W^{j}\dot{q}^{i}\dot{q}^{l}T_{ij}^{k}\Gamma_{lk}^{m}-\dot{q}^{i}\dot{q}^{l}W^{j}\frac{\partial T_{ij}^{m}}{\partial q^{l}} \right)\frac{\partial}{\partial q^{m}}=\nabla^{nh}_{\dot{c}}T^{nh}(W,\dot{c})
	\end{equation*}
	Using that $c$ is a geodesic, so
	\begin{equation*}
	\ddot{q}^{i}=-\Gamma_{jk}^{i}\dot{q}^{j}\dot{q}^{k},
	\end{equation*}
	and the identity
	\begin{equation*}
	T_{ik}^{m}=-T_{ki}^{m}
	\end{equation*}
	we get
	\begin{equation}
	\begin{split}
	\nabla^{nh}_{\dot{c}}T^{nh}(W,\dot{c})=\left( \dot{W}^{i}\dot{q}^{j}T_{ij}^{m}-W^{i}\Gamma_{lk}^{j}\dot{q}^{l}\dot{q}^{k}T_{ij}^{m}\right.&-W^{i}\dot{q}^{j}\frac{\partial T_{ji}^{m}}{\partial q^{l}}\dot{q}^{l}\\
	&\left.-W^{i}\dot{q}^{j}T_{ji}^{k}\dot{q}^{l}\Gamma_{lk}^{m} \right)\frac{\partial}{\partial q^{m}},
	\end{split}
	\end{equation}
	which is what we expected. Hence, by Proposition \ref{completeconnection} and the remark following it we proved that $W$ is a nonholonomic Jacobi field if and only if it satisfies equations \eqref{nhjacobi}.
	
\end{proof}

The coordinate expression of the nonholonomic Jacobi equation is still a second-order differential equation. Indeed, the local expression of \eqref{nhjacobi} is
\begin{equation}\label{coordinatejacobi}
	\left( \ddot{W}^{k}+\dot{q}^{i}\dot{q}^{j}W^{l}\frac{\partial \Gamma^{k}_{ij}}{\partial q^{l}}+2\dot{q}^{i}\dot{W}^{j}\Gamma^{k}_{ij}-\dot{q}^{i}\dot{W}^{j}T_{ij}^{k} \right)\frac{\partial}{\partial q^{k}}.
\end{equation}

Equation \eqref{nhjacobi} is called the \textit{nonholonomic Jacobi equation} for the nonholonomic geodesic problem.

\begin{example}
Recall the nonholonomic particle given by $$L(x,y,z,\dot{x},\dot{y},\dot{z})=\frac{1}{2}(\dot{x}^2 +\dot{y}^2 +\dot{z}^2)$$ and subjected to the nonholonomic constraint $\dot{z}-y \dot{x}=0$.

As we have seen before, nonholonomic Jacobi fields may be obtained using two different geometric frameworks: either as the trajectories of the lifted nonholonomic system $\Gamma_{(L_{g},\D^{c})}$ or as the solution of the nonholonomic Jacobi equation \eqref{nhjacobi}. Let us explore both these characterizations in this particular example.

\begin{enumerate}
	\item[(i)] We are going to obtain the Lagrange-d'Alembert equations for the nonholonomic system $L_{g^{c}}$ with constraint distribution $\D^{c}$.
	
	The Lagrangian function is
	{\begin{equation*}
	L_{g^{c}}(q,r,\dot{q},\dot{r})=\dot{x}\dot{u}+\dot{y}\dot{v}+\dot{z}\dot{w}.
	\end{equation*}}
	{where $q = (x, y , z)$, $r = (u, v, w)$}
	and the lifted distribution $\D^{c}$ is given by the span of the vectors
	\begin{equation*}
	\D^{c}=\left\langle \left\{ \frac{\partial}{\partial u}+y\frac{\partial}{\partial w}, \frac{\partial}{\partial v}, \frac{\partial}{\partial x}+y\frac{\partial}{\partial z}+v\frac{\partial}{\partial w}, \frac{\partial}{\partial y} \right\} \right\rangle
	\end{equation*}
	with annihilator $(\D^{c})^{o}$
	\begin{equation*}
	(\D^{c})^{o}=\left\langle \left\{ -y dx+dz, -v dx-ydu+dw \right\} \right\rangle.
	\end{equation*}
	Hence, the new nonholonomic constraints are $\dot{z}-y\dot{x}=0$ and $\dot{w}-v\dot{x}-y\dot{u}=0$. The Lagrange-d'Alembert equations are then
	\begin{equation*}
	\begin{split}
	& \ddot{x}=-y\lambda_{2} \\
	& \ddot{y}=0 \\
	& \ddot{z}=\lambda_{2} \\
	& \dot{z}-y\dot{x}=0,
	\end{split}
	\quad
	\begin{split}
	& \ddot{u}=-y\lambda_{1}-v\lambda_{2}\\
	& \ddot{v}=0 \\
	& \ddot{w}=\lambda_{1} \\
	& \dot{w}-v\dot{x}-y\dot{u}=0,
	\end{split}
	\end{equation*}
	and solving for the Lagrange multipliers' $\lambda_{1}$ and $\lambda_{2}$, we obtain
	\begin{equation*}
		\lambda_{1}=\frac{(\dot{u}\dot{y}+\dot{x}\dot{v})(1+y^2)-2yv\dot{x}\dot{y}}{(1+y^2)^{2}}, \quad \lambda_{2}=\frac{\dot{x}\dot{y}}{1+y^2}.
	\end{equation*} 
	
	\item[(ii)] We will compute the nonholonomic Jacobi equation using the local expression deduced in \eqref{coordinatejacobi}. The only non-vanishing Chrystoffel symbols relative to the nonholonomic connection $\nabla^{nh}$ are
	\begin{equation*}
		\Gamma_{yx}^{x}=\frac{2y}{(1+y^2)^2}, \quad \Gamma_{yx}^{z}=\Gamma_{yz}^{x}=\frac{y^{2}-1}{(1+y^2)^2}, \quad \Gamma_{yz}^{z}=-\frac{2y}{(1+y^2)^2},
	\end{equation*}
	which implies that the non-vanishing torsion entries are
	\begin{equation*}
		T_{yx}^{x}=\frac{2y}{(1+y^2)^2}, \quad T_{yx}^{z}=T_{yz}^{x}=\frac{y^{2}-1}{(1+y^2)^2}, \quad T_{yz}^{z}=-\frac{2y}{(1+y^2)^2},
	\end{equation*}
	along with the corresponding skew-symmetric entries. If the vector field $W$ is given by
	\begin{equation*}
		W=u\frac{\partial}{\partial x}+v\frac{\partial}{\partial y}+w\frac{\partial}{\partial z},
	\end{equation*}
	then Jacobi equation together with the constraint that $\dot{W}\in\D^{c}$ gives
	\begin{equation}
	\begin{split}
	& \ddot{u}+v\left(\dot{x}\dot{y}\frac{\partial \Gamma^{x}_{yx}}{\partial y}+\dot{z}\dot{y}\frac{\partial \Gamma^{x}_{yz}}{\partial y}\right)+2(\dot{u}\dot{y}\Gamma^{x}_{yx}+\dot{w}\dot{y}\Gamma^{x}_{yz})-(\dot{u}\dot{y}T^{x}_{yx}+\dot{w}\dot{y}T^{x}_{yz})=0 \\
	& \ddot{v}=0 \\
	& \ddot{w}+v\left(\dot{x}\dot{y}\frac{\partial \Gamma^{z}_{yx}}{\partial y}+\dot{z}\dot{y}\frac{\partial \Gamma^{z}_{yz}}{\partial y}\right)+2(\dot{u}\dot{y}\Gamma^{z}_{yx}+\dot{w}\dot{y}\Gamma^{z}_{yz})-(\dot{u}\dot{y}T^{z}_{yx}+\dot{w}\dot{y}T^{z}_{yz})=0 \\
	& \dot{w}-v\dot{x}-y\dot{u}=0.
	\end{split}
	\end{equation}
	The fact that $W$ is a vector field along a nonholonomic geodesic satisfying $\dot{z}=y\dot{x}$ simplifies the equation. Moreover, since $T_{yx}^{x}=\Gamma_{yx}^{x}$, $T_{yx}^{z}=T_{yz}^{x}=\Gamma_{yx}^{z}=\Gamma_{yz}^{x}$ and $T_{yz}^{z}=\Gamma_{yz}^{z}$ simplifies even more the equation that reduces to
	\begin{equation}
	\begin{split}
	& \ddot{u}+v\dot{x}\dot{y}\left(\frac{\partial \Gamma^{x}_{yx}}{\partial y}+y\frac{\partial \Gamma^{x}_{yz}}{\partial y}\right)+\dot{u}\dot{y}\Gamma^{x}_{yx}+\dot{w}\dot{y}\Gamma^{x}_{yz}=0 \\
	& \ddot{v}=0 \\
	& \ddot{w}+v\dot{x}\dot{y}\left(\frac{\partial \Gamma^{z}_{yx}}{\partial y}+y\frac{\partial \Gamma^{z}_{yz}}{\partial y}\right)+\dot{u}\dot{y}\Gamma^{z}_{yx}+\dot{w}\dot{y}\Gamma^{z}_{yz}=0 \\
	& \dot{w}-v\dot{x}-y\dot{u}=0.
	\end{split}
	\end{equation}
	It is easy to see now that both approaches coincide.	
\end{enumerate}

\end{example}

\begin{example}
	Let us consider again the nonholonomic particle. We proved before that $\frac{\partial}{\partial z}$ was a Jacobi field. Let us check that it satisfies the nonholonomic Jacobi equation.
	
	In fact, since the component functions of $\frac{\partial}{\partial z}$ in the coordinate basis are constant, the local expression of the left-hand side of the nonholonomic Jacobi equation reduces to
	\begin{equation*}
		\nabla^{nh}_{\dot{c}}\nabla^{nh}_{\dot{c}}W+\nabla^{nh}_{\dot{c}}T^{nh}(W,\dot{c})+R^{nh}(W,\dot{c})\dot{c}=\dot{q}^{i}\dot{q}^{j}\frac{\partial \Gamma^{k}_{ij}}{\partial z}\frac{\partial}{\partial q^{k}},
	\end{equation*}
	where $(q^{i}(t))$ are the coordinate expression of a fixed geodesic. However, since the only non-vanishing Chrystoffel symbols relative to the nonholonomic connection are
	\begin{equation*}
		\Gamma_{yx}^{x}=\frac{2y}{(1+y^2)^2}, \quad \Gamma_{yx}^{z}=\Gamma_{yz}^{x}=\frac{y^{2}-1}{(1+y^2)^2}, \quad \Gamma_{yz}^{z}=-\frac{2y}{(1+y^2)^2},
	\end{equation*}
	there is no dependence on the coordinate $z$, hence the expression above vanishes, i.e., the nonholonomic Jacobi equations is satisfied.
\end{example}



\appendix

\section{Review on Complete and Vertical lifts}\label{appa}

In this section, we will review some constructions on the theory of complete and vertical lifts in the tangent bundle (for more details, see \citep*{LR89} or \citep*{YaIs73}).

Let $\tau_{Q}:TQ \rightarrow Q$ be the canonical projection of the tangent bundle.

Recall that the complete and vertical lifts of a function $f\in C^{\infty}(Q)$ are defined by
\begin{equation}\label{function:lifts}
f^{c}(v)=\langle df(q),v \rangle, \quad f^{v}(v)=f\circ\tau_{Q}(v), \ \ v\in T_{q} Q.
\end{equation}

{In other words, $f^{c}$ is the fiberwise linear function on $TQ$ induced by the $1$-form $df$}. Hence, the complete lift of $f$ may be expressed on natural coordinates as $$f^{c}(q^{i},v^{i})=\frac{\partial f}{\partial q^{i}}v^{i}.$$

We will also use the \textit{complete lift} of a vector field $X\in\mathfrak{X}(Q)$, which is the vector field $X^{c}\in\mathfrak{X}(TQ)$ satisfying
\begin{equation}\label{completelift}
X^{c}(f\circ\tau_{Q})=X(f)\circ\tau_{Q}, \text{  and  } X^{c}(\hat{\alpha})=\widehat{\mathcal{L}_{X}\alpha},
\end{equation}
for any $f\in C^{\infty}(Q)$, $\alpha\in\Omega^{1}(Q)$ and where $\hat{\alpha}\in C^{\infty}(TQ)$ is the associated fiberwise linear function given by
\begin{equation*}
	\hat{\alpha}(v)=\langle \alpha(\tau_{Q}(v)), v \rangle, \quad v\in TQ.
\end{equation*}

In what follows $\Le_{X}$ denotes the Lie derivative with respect to $X$.

We may also introduce the fiberwise quadratic function associated to a $(0,2)$-tensor $T$ on $Q$ denoted by $T^{\mathfrak{q}}:TQ\rightarrow \R$ and defined by
\begin{equation*}
	T^{\mathfrak{q}}(v)=T_{\tau_{Q}(v)} (v,v), \quad v\in TQ.
\end{equation*}

Using equation \eqref{completelift} one may prove
\begin{lemma}
	If $X$ is a vector field on $Q$ and $T$ is a $(0,2)$-tensor on $Q$ then
	\begin{equation}\label{completequadratic}
	X^{c}(T^{\mathfrak{q}})=(\Le_{X}T)^{\mathfrak{q}}.
	\end{equation}
\end{lemma}

\begin{proof}
	It is sufficient to prove the result for $T=\alpha\otimes\beta$, with $\alpha$ and $\beta$ being 1-forms on $Q$. In this case,
	\begin{equation*}
		T^{\mathfrak{q}}=\hat{\alpha}\cdot\hat{\beta}
	\end{equation*}
	and, using \eqref{completelift}, it follows that
	\begin{equation*}
	X^{c}(T^{\mathfrak{q}})=\widehat{\mathcal{L}_{X}\alpha}\cdot\hat{\beta}+\hat{\alpha}\cdot\widehat{\mathcal{L}_{X}\beta}.
	\end{equation*}
	This implies that
	\begin{equation*}
	X^{c}(T^{\mathfrak{q}})=\left[\mathcal{L}_{X}\alpha\otimes\beta+\alpha\otimes\mathcal{L}_{X}\beta\right]^{\mathfrak{q}}=(\mathcal{L}_{X} T)^{\mathfrak{q}}.
	\end{equation*}
\end{proof}

The \textit{vertical lift} of a vector field is the vector field $X^{v}\in\mathfrak{X}(TQ)$ satisfying
\begin{equation}
X^{v}(f\circ\tau_{Q})=0, \text{  and  } X^{v}(\hat{\alpha})=\langle \alpha,X \rangle \circ \tau_{Q}.
\end{equation}

Similarly we can define the \textit{complete lift} of a $k$-form $\alpha\in\Omega^{k}(Q)$ to be the $k$-form $\alpha^{c}\in\Omega^{k}(TQ)$ defined by
\begin{equation}\label{1formcomplete}
\alpha^{c}(X_{1}^{c},\ldots,X_{k}^{c})=(\alpha(X_{1},\ldots,X_{k}))^{c},
\end{equation}
where $X_{i}\in \mathfrak{X}(Q)$. The expression above uniquely defines $\alpha^{c}$ and, moreover,
\begin{equation}\label{differential:completelift}
	d\alpha^{c}=(d\alpha)^{c}
\end{equation}
and
\begin{equation}\label{inner:complete}
i_{X^{c}}\alpha^{c}=(i_{X}\alpha)^{c}.
\end{equation}
On the other hand, the \textit{vertical lift} of a $k$-form is simply the pullback by $\tau_{Q}$, i.e.,
\begin{equation}\label{1formvertical}
\alpha^{v}=(\tau_{Q})^{*}\alpha.
\end{equation}

In coordinates, the expressions of the complete and vertical lifts of $X=X^{i}\frac{\partial}{\partial q^{i}}$ and $\alpha=\alpha_{i} dq^{i}$ are
\begin{equation}
\begin{split}
X^{c}=X^{i}\frac{\partial}{\partial q^{i}}+\frac{\partial X^{i}}{\partial q^{j}}v^{j} \frac{\partial}{\partial v^{i}}, \quad  & \alpha^{c}=\frac{\partial \alpha_{i}}{\partial q^{j}}v^{j} dq^{i} + \alpha_{i} dv^{i} \\
X^{v}=X^{i}\frac{\partial}{\partial v^{i}}, \quad & \alpha^{v}=\alpha_{i} dq^{i}.
\end{split}
\end{equation}

To end this section, we will recall some useful identities satisfied by vertical and complete lifts. For any one form $\alpha\in \Omega^{1}(Q)$, one has that
\begin{equation}\label{alphalift}
	 \alpha^{v}(Y^{c})=(\alpha(Y))^{v}, \quad \alpha^{v}(Y^{v})=0.
\end{equation}

The Lie bracket of vector fields on complete and vertical lifts satisfies the following relations
\begin{equation}\label{liebralift}
	\begin{split}
	[X^{c},Y^{c}]&=[X,Y]^{c}, \quad [X^{v},Y^{v}]=0,\\ [X^{c},Y^{v}]&=[X^{v},Y^{c}]=[X,Y]^{v}.
	\end{split}
\end{equation}

On the other hand, if $S:TTQ\rightarrow TTQ$ is the vertical endomorphism in $Q$ then complete and vertical lifts satisfy
\begin{equation}\label{Slifts}
	SX^{c}=X^{v}, \quad SX^{v}=0, \quad \text{for} \ X\in\mathfrak{X}(Q)
\end{equation}

Finally, if $S^{*}:T^{*}TQ\rightarrow T^{*}TQ$ is the dual morphism of $S$, then
\begin{equation}\label{S*lifts}
S^{*}\alpha^{c}=\alpha^{v}, \quad S^{*}\alpha^{v}=0, \quad S^{*}(d\hat{\alpha})=\alpha^{v}, \quad \text{for} \ \alpha\in\Omega^{1}(Q).
\end{equation}


\section{Complete and vertical lifts in Riemannian geometry}\label{appb}

The main result of this section is to prove how a pseudo-Riemannian metric $h$ is related to the Poincaré-Cartan two-form induced by this metric. In particular, the action of the later on complete and vertical lifts may be expressed with objects that depend only on the metric structure.

Recall that when we have a manifold $Q$ and a Lagrangian $L$ on its tangent bundle, the Poincaré-Cartan one-form is defined to be $\theta_{L}=S^{*}dL$. 

In agreement with the previous notation, whenever we are given a $(0,2)$-tensor $h$ on $Q$, we will denote by $L_{h}:TQ\rightarrow\R$ the Lagrangian function associated with $h$ defined by
\begin{equation*}
L_{h}(v)=\frac{1}{2}h(v,v), \quad v\in TQ.
\end{equation*}

Also we will denote by
\begin{equation*}
	\flat_{h}:TQ\rightarrow T^{*}Q,
\end{equation*}
the musical isomorphism associated with $h$ by
$\flat_h(X)(Y)=h(X, Y)$ for all $X, Y\in \mathfrak{X}(Q)$.

We will generalize the fundamental formula of Riemannian geometry 

\begin{lemma}\label{liedelemmma}
	Let $h$ be a symmetric non-degenerate $(0,2)$-tensor and $\nabla^{h}$ the Levi-Civita connection with respect to $h$. Then the Lie derivative of $h$ satisfies
	\begin{equation}\label{FFRG}
	\Le_{X}h (Y,Z)=2h(\nabla^{h}_{Y}X,Z)-d(\flat_{h}(X))(Y,Z), \quad X,Y,Z \in \mathfrak{X}(Q).
	\end{equation}
\end{lemma}

\begin{proof}
	By definition of Lie derivative one has that
	\begin{equation*}
	\Le_{X}h (Y,Z)=X(h(Y,Z))-h([X,Y],Z)-h(Y,[X,Z]).
	\end{equation*}
	Using the fact that the Levi-Civita connection $\nabla^{h}$ is symmetric and compatible with the metric $h$ one gets
	\begin{equation}\label{LXh}
	\Le_{X}h (Y,Z)=h(\nabla^{h}_{Y}X,Z)+h(Y,\nabla^{h}_{Z}X).
	\end{equation}
	Also from definition of differential of a one-form we know that
	\begin{equation*}
	d(\flat_{h}(X))(Y,Z)=Y(h(X,Z))-Z(h(X,Z))-h(X,[Y,Z]).
	\end{equation*}
	It is not difficult to show using the Koszul's formula for $\nabla^{h}$ that
	\begin{equation*}
	h(Y,\nabla^{h}_{Z}X)-h(\nabla^{h}_{Y}X,Z)=-d(\flat_{h}(X))(Y,Z)
	\end{equation*}
	and plugging in the last equation into \eqref{LXh}, we get the desired formula for $\Le_{X}h (Y,Z)$.
\end{proof}

During the remaining of this section, we will denote just by $\nabla$ the Levi-Civita connection with respect to a symmetric non-degenerate $(0,2)$-tensor, whenever it is clear from the context to which tensor it is associated.

Now we will see how complete and vertical lifts act on metric Lagrangians.

\begin{lemma}\label{cvlift}
	Let $h$ be a pseudo-Riemannian metric and $L_{h}$ its associate Lagrangian. Given $X\in \mathfrak{X}(Q)$ we have that
	\begin{equation}
	X^{c}(L_{h})=L_{\Le_{X}h}, \quad X^{v}(L_{h})=\widehat{\flat_{h}(X)},
	\end{equation}
	where $L_{\Le_{X}h}: TQ\rightarrow \R$ denotes the Lagrangian function associated to the $(0,2)$-tensor $\Le_{X}h$.
\end{lemma}

\begin{proof}
	Let us prove the result on natural coordinates. Let $(q^{i})$ be coordinates on $Q$ and $(q^{i},v^{i})$ be the natural coordinates on $TQ$. Let $X=X^{i}\frac{\partial}{\partial q^{i}}$ and $L_{h}=\frac{1}{2}h_{ij}v^{i}v^{j}$. Then
	\begin{equation*}
		X^{v}(L_{h})=X^{i}h_{ij}v^{j}=\widehat{\flat_{h}(X)}
	\end{equation*}
	and
	\begin{equation*}
		X^{c}(L_{h})=\frac{1}{2}X^{k}\frac{\partial h_{ij}}{\partial q^{k}}v^{i}v^{j}+v^{i}\frac{\partial X^{k}}{\partial q^{i}}h_{kj}v^{j}.
	\end{equation*}
	On the other hand
	\begin{equation*}
		L_{\Le_{X}h}=\frac{1}{2}\left( X^{k}\frac{\partial h_{ij}}{\partial q^{k}}+h_{kj}\frac{\partial X^{k}}{\partial q^{i}}+h_{ik}\frac{\partial X^{k}}{\partial q^{j}} \right)v^{i}v^{j}.
	\end{equation*}
	Since $h$ is symmetric and changing indices $i\leftrightarrow j$ in the last term above, we get
	\begin{equation*}
		L_{\Le_{X}h}=\frac{1}{2} X^{k}\frac{\partial h_{ij}}{\partial q^{k}}v^{i}v^{j}+h_{kj}\frac{\partial X^{k}}{\partial q^{i}} v^{i}v^{j},
	\end{equation*}
	which equals $X^{c}(L_{h})$.
\end{proof}

The fundamental formula of Riemannian geometry allows us to express the Lagrangian function $L_{\Le_{X}h}$ introduced before in terms of a new Lagrangian function associated with the $(0,2)$-tensor
\begin{equation}\label{nablatensor}
(\nabla^{h} X) (Y,Z)=h(\nabla^{h}_{Y}X,Z),
\end{equation}
where $\nabla^{h}$ is the Levi-Civita connection with respect to $h$. Indeed, given any $Y\in\mathfrak{X}(Q)$, by skew-symmetry of the exterior derivative one has that 
\begin{equation*}
\begin{split}
L_{\Le_{X}h}\circ Y & =\frac{1}{2}\Le_{X}h (Y,Y)=h(\nabla^{h}_{Y}X,Y)-\frac{1}{2}d(\flat_{h}(X))(Y,Y) \\
& = (\nabla^{h} X)(Y,Y)=2L_{(\nabla^{h} X)}\circ Y.
\end{split}
\end{equation*}

\begin{lemma}
	Let $h$ be a {pseudo-Riemanninan metric} on $Q$, $L_{h}$ is its associated Lagrangian function and $\omega_{L_{h}}=-d\theta_{L_{h}}$ the corresponding Poincaré-Cartan 2-form. If $\flat_{\omega_{L_{h}}}$ and $\flat_{h}$ denote the musical isomorphisms associated to the symplectic form and to the tensor $h$, respectively, then for every $X\in\mathfrak{X}(Q)$ the Poincaré-Cartan 1-form acts on vertical and complete lifts of $X$ according to
	\begin{equation}\label{thetalift}
		 \theta_{L_{h}}(X^{v})=0, \quad \theta_{L_{h}}(X^{c})=\widehat{\flat_{h}(X)}
	\end{equation}
	and the Poincaré-Cartan 2-form acts according to
	\begin{equation}
		\begin{split}
			&\omega_{L_{h}}(X^{c},Y^{c})=d(\widehat{\flat_{h}(X)})(Y^{c})-2\theta_{L_{(\nabla^{h} X)}}(Y^{c}), \\
			&\omega_{L_{h}}(X^{c},Y^{v})=d(\widehat{\flat_{h}(X)})(Y^{v}) , \quad \omega_{L_{h}}(X^{v},Y^{v})=0,
		\end{split}
	\end{equation}
	where $(\nabla^{h} X)$ is the (0,2)-tensor defined in (\ref{nablatensor}). Hence, we may also write
	\begin{equation}\label{omegalift}
		\flat_{\omega_{L_{h}}}(X^{v})=-(\flat_{h}(X))^{v}, \quad \flat_{\omega_{L_{h}}}(X^{c})=d(\widehat{\flat_{h}(X)})-2\theta_{L_{(\nabla^{h} X)}}.
	\end{equation}
 \end{lemma}

\begin{proof}
	Recalling the definition of $\theta_{L_{h}}$, the formulas in the statement are rewritten as
	\begin{equation*}
		\theta_{L_{h}}(X^{v})=S^{*}(dL_{h})(X^{v}), \quad \theta_{L_{h}}(X^{c})=S^{*}(dL_{h})(X^{c}),
	\end{equation*}
	respectively. Now applying \eqref{Slifts} we immediately prove that $\theta_{L_{h}}(X^{v})=0$ and
	\begin{equation*}
		\theta_{L_{h}}(X^{c})=X^{v}(L_{h}).
	\end{equation*}
	By the previous Lemma we conclude $\theta_{L_{h}}(X^{c})=\widehat{\flat_{h}(X)}$.

	Choosing an arbitrary $Y\in\mathfrak{X}(Q)$, we will now evaluate the symplectic form over complete and vertical lifts in order to find the desired formulas for $\flat_{\omega_{L_{h}}}(X^{c})$ and $\flat_{\omega_{L_{h}}}(X^{v})$.
	
	Using that $\omega_{L_{h}}$ is an exact symplectic form and the characterization of the exterior derivative of a 1-form we get
	\begin{equation*}
		\omega_{L_{h}}(X^{c},Y^{c})=-X^{c}\left(\theta_{L_{h}}(Y^{c})\right)+Y^{c}\left(\theta_{L_{h}}(X^{c})\right)+\theta_{L_{h}}([X^{c},Y^{c}]).
	\end{equation*}
	Using equations \eqref{thetalift} we have just proved and the formulas in \eqref{liebralift} we get
	\begin{equation*}
		\omega_{L_{h}}(X^{c},Y^{c})=-X^{c}\left(\widehat{\flat_{h}(Y)}\right)+Y^{c}\left(\widehat{\flat_{h}(X)}\right)+\widehat{\flat_{h}([X,Y])}.
	\end{equation*}
	Applying now the definition of complete lift over fiberwise linear functions
	\begin{equation*}
	\omega_{L_{h}}(X^{c},Y^{c})=-\widehat{\Le_{X}\flat_{h}(Y)}+\widehat{\Le_{Y}\flat_{h}(X)}+\widehat{\flat_{h}([X,Y])}.
	\end{equation*}
	Note that
	\begin{equation*}
		\flat_{h}([X,Y])(Z)-\Le_{X}(\flat_{h}(Y))(Z)=-\Le_{X}h(Y,Z).
	\end{equation*}
	Hence, the right-hand side of the above equation may be rewritten using the musical isomorphism associated to the $(0,2)$-tensor $\Le_{X}h$ which we denote by $\flat_{\Le_{X}h}$. So we deduce that
	\begin{equation*}
	\omega_{L_{h}}(X^{c},Y^{c})=-\widehat{\flat_{\Le_{X}h}(Y)}+\widehat{\Le_{Y}\flat_{h}(X)}.
	\end{equation*}
	Now using again the relations in \eqref{thetalift} and the definition of complete lift
	\begin{equation*}
	\omega_{L_{h}}(X^{c},Y^{c})=-\theta_{L_{\Le_{X}h}}(Y^{c})+Y^{c}(\widehat{\flat_{h}(X)}).
	\end{equation*}
	Using that $L_{\Le_{X}h}=2L_{(\nabla^{h} X)}$ and rewriting the last term above, we finally get
	\begin{equation*}
	\omega_{L_{h}}(X^{c},Y^{c})=-2\theta_{L_{(\nabla^{h} X)}}(Y^{c})+d(\widehat{\flat_{h}(X)})(Y^{c}).
	\end{equation*}
	Proceeding analogously in the other cases we find that
	\begin{equation*}
	\omega_{L_{h}}(X^{c},Y^{v})=-X^{c}\left(\theta_{L_{h}}(Y^{v})\right)+Y^{v}\left(\theta_{L_{h}}(X^{c})\right)+\theta_{L_{h}}([X^{c},Y^{v}]).
	\end{equation*}
	Using \eqref{thetalift} and \eqref{liebralift}
	\begin{equation*}
	\omega_{L_{h}}(X^{c},Y^{v})=Y^{v}\left(\widehat{\flat_{h}(X)}\right)+\theta_{L_{h}}([X,Y]^{v}).
	\end{equation*}
	Again using \eqref{thetalift} we conclude
	\begin{equation*}
	\omega_{L_{h}}(X^{c},Y^{v})=d\left(\widehat{\flat_{h}(X)}\right)(Y^{v}).
	\end{equation*}
	Note also that $\theta_{L_{(\nabla^{h} X)}}(Y^{v})=0$ is also implied by \eqref{thetalift}. Therefore we have concluded the proof of the expression for $\flat_{\omega_{L_{h}}}(X^{c})$.
	
	Let us now prove the expression for $\flat_{\omega_{L_{h}}}(X^{v})$. Let us use the same strategy and compute 
	\begin{equation*}
	\omega_{L_{h}}(X^{v},Y^{c})=-X^{v}\left(\theta_{L_{h}}(Y^{c})\right)+Y^{c}\left(\theta_{L_{h}}(X^{v})\right)+\theta_{L_{h}}([X^{v},Y^{c}]).
	\end{equation*}
	Analogously,
	\begin{equation*}
	\omega_{L_{h}}(X^{v},Y^{c})=-X^{v}\left(\widehat{\flat_{h}(Y)}\right).
	\end{equation*}
	Now, using the definition of vertical lift
	\begin{equation*}
	\omega_{L_{h}}(X^{v},Y^{c})=-\flat_{h}(Y)(X)\circ\tau_Q.
	\end{equation*}
	By symmetry of $h$, we may rewrite the last line as
	\begin{equation*}
	\omega_{L_{h}}(X^{v},Y^{c})=-\flat_{h}(X)(Y)\circ\tau_Q.
	\end{equation*}
	Notice that the right-hand side of the previous equation is nothing more than the vertical lift of the function $\flat_{h}(X)(Y)$. Using \eqref{alphalift}, we finally get
	\begin{equation*}
	\omega_{L_{h}}(X^{v},Y^{c})=-(\flat_{h}(X))^{v}(Y^{c}).
	\end{equation*}
	At last, we need to check how the symplectic form acts on vertical lifts. However, note that using \eqref{thetalift} and \eqref{liebralift} then the expression
	\begin{equation*}
	\omega_{L_{h}}(X^{v},Y^{v})=-X^{v}\left(\theta_{L_{h}}(Y^{v})\right)+Y^{v}\left(\theta_{L_{h}}(X^{v})\right)+\theta_{L_{h}}([X^{v},Y^{v}])
	\end{equation*}
	vanishes for all $X$ and $Y$, as it is the case of $(\flat_{h}(X))^{v}(Y^{v})$ again by \eqref{alphalift}, which finishes the proof.
\end{proof}


\section{The canonical involution and the complete lift of a Lagrangian system of kinetic type}\label{appc}

In this appendix, we will review some results related with the canonical involution on the double tangent bundle of a manifold \citep*{Tulcz} and the complete lift of a regular Lagrangian system of kinetic type.

Let $Q$ be a smooth manifold of dimension $n$, $\tau_{Q}:TQ\rightarrow Q$ the canonical projection and $TTQ$ the double tangent bundle to $Q$. Then, $TTQ$ admits two vector bundle structures.

The first vector bundle structure is the canonical one with vector bundle projection $\tau_{TQ}:TTQ\rightarrow TQ$.

For the second vector bundle structure, the vector bundle projection is just the tangent map to $\tau_{Q}$, that is, $T\tau_{Q}:TTQ\rightarrow TQ$ and the addition operation on the fibers is just the tangent map $T(+):TTQ\times_{TQ} TTQ\rightarrow TTQ$ of the addition operation $(+):TQ\times_{Q}TQ\rightarrow TQ$ on the fibers of $\tau_{Q}$.

The canonical involution $\kappa_{Q}:TTQ\rightarrow TTQ$ is a vector bundle isomorphism (over the identity of $TQ$) between the two previous vector bundles. In fact, $\kappa_{Q}$ is characterized by the following condition: let $\Phi:U\subseteq \R^{2}\rightarrow Q$ be a smooth map, with $U$ an open subset of $\R^{2}$
\begin{equation*}
(t,s)\mapsto \Phi(t,s)\in Q.
\end{equation*}
Then,
\begin{equation}\label{def:can:invol}
\kappa_{Q}\left( \frac{d}{dt}\frac{d}{ds} \Phi(t,s) \right)= \frac{d}{ds}\frac{d}{dt} \Phi(t,s).
\end{equation}
So, we have that $\kappa_{Q}$ is an involution of $TTQ$, that is, $\kappa_{Q}^{2}=id_{TTQ}$.

In fact, if $(q^{i},\dot{q}^{i})$ are canonical fibred coordinates on $TQ$ and $(q^{i},\dot{q}^{i},v^{i},\dot{v}^{i})$ are the corresponding local fibred coordinates on $TTQ$ then
\begin{equation}\label{def:local:can:invol}
\kappa_{Q}(q^{i},\dot{q}^{i},v^{i},\dot{v}^{i})=(q^{i},v^{i},\dot{q}^{i},\dot{v}^{i}).
\end{equation}
$\kappa_{Q}$ may be characterized in a more intrinsic way, using the theory of complete and vertical lifts to $TQ$.

Indeed, if $X:Q\rightarrow TQ$ is a vector field on $Q$ then
\begin{equation}\label{int:def:can:invol}
\kappa_{Q}\circ X^{c}=TX, \quad \kappa_{Q}\circ X^{v}=\widetilde{X}^{v},
\end{equation}
where $TX:TQ\rightarrow TTQ$ is the tangent map to $X$ (a section of the vector bundle $T\tau_{Q}$) and $\widetilde{X}^{v}:TQ\rightarrow TTQ$ is the section of the vector bundle $T\tau_{Q}$ given by
\begin{equation*}
	\widetilde{X}^{v}(u)=(T_{q}0)(u)+X^{v}(0(q)), \quad u\in T_{q}Q,
\end{equation*}
with $0:Q\rightarrow TQ$ the zero section.

Note that, from \eqref{def:can:invol}, it follows that
\begin{equation}\label{diff:can:invol}
	\begin{split}
		& T\kappa_{Q}\circ (X^{c})^{c}=(X^{c})^{c}\circ \kappa_{Q}, \quad T\kappa_{Q} \circ (X^{v})^{v}=(X^{v})^{v}\circ \kappa_{Q}, \\
		& T\kappa_{Q}\circ (X^{c})^{v}=(X^{v})^{c}\circ \kappa_{Q}, \quad T\kappa_{Q} \circ (X^{v})^{c}=(X^{c})^{v}\circ \kappa_{Q},
	\end{split}
\end{equation}
for $X\in\mathfrak{X}(Q)$.

As a consequence, we also deduce that
\begin{equation}\label{codiff:can:invol}
\begin{split}
& \kappa_{Q}^{*} ((\alpha^{c})^{c})=(\alpha^{c})^{c}\circ \kappa_{Q}, \quad \kappa_{Q}^{*}( (\alpha^{v})^{v})=(\alpha^{v})^{v}\circ \kappa_{Q}, \\
& \kappa_{Q}^{*}((\alpha^{c})^{v})=(\alpha^{v})^{c}\circ \kappa_{Q}, \quad \kappa_{Q}^{*} ((\alpha^{v})^{c})=(\alpha^{c})^{v}\circ \kappa_{Q},
\end{split}
\end{equation}
for $\alpha\in\Omega^{1}(Q)$.

Now suppose that $g$ is a Riemannian metric on $Q$ and that $L_{g}:TQ\rightarrow \R$ is the Lagrangian function of kinetic type induced by $g$ (see Appendix B). Then, we may consider the complete lift $g^{c}$ of $g$ \citep*{YaIs73}. It is not a Riemannian metric on $TQ$ but a pseudo-Riemannian metric of signature $(n,n)$. In fact $g^{c}$ is characterized by the following conditions
\begin{equation}\label{def:complete:lift:g}
\begin{split}
& g^{c}(X^{c},Y^{c})=(g(X,Y))^{c}, \\
& g^{c}(X^{c},Y^{v})=g^{c}(X^{v},Y^{c})=(g(X,Y))^{v}, \\
& g^{c}(X^{v},Y^{v})=0,
\end{split}
\end{equation}
for $X,Y\in\mathfrak{X}(Q)$.

If $(q^{i},\dot{q}^{i})$ are local coordinates on $TQ$ and the local expression of the Riemannian metric $g$ is $g=g_{ij} dq^{i}\otimes dq^{j}$ then the local expression of its complete lift is
\begin{equation*}
g^{c}(q^{i},\dot{q}^{i})=\dot{q}^{k}\frac{\partial g_{ij}}{\partial q^{k}} dq^{i}\otimes dq^{j}+ g_{ij}dq^{i}\otimes d\dot{q}^{j}+g_{ij}d\dot{q}^{i}\otimes dq^{j}.
\end{equation*}

Anyway, we may consider the Lagrangian function $L_{g^{c}}:TTQ\rightarrow \R$ on $TTQ$ induced by the pseudo-Riemannian metric $g^{c}$ on $TQ$. Then, the relation of the previous construction with the canonical involution is given by the following result:

\begin{lemma} \label{liftedLagrangianlemma}
	We have that the Lagrangian function $L_{g^{c}}:TTQ\rightarrow \R$ is regular and satisfies the following equation
	\begin{equation}
	L_{g^{c}}=L_{g}^{c}\circ \kappa_{Q}.
	\end{equation}
\end{lemma}

\begin{proof}
	The Lagrangian function $L_{g^{c}}$ is regular since its Hessian matrix is the tensor $g^{c}$ which is a non-degenerate tensor. In fact, it is a pseudo-Riemannian metric. 
	
	If $Z\in T_{u}(TQ)$, with $u\in T_{q} Q$, then it is easy to see that there exist vector fields $X,Y\in \mathfrak{X}(Q)$ such that
	\begin{equation*}
		Z=X^{c}(u)+Y^{v}(u).
	\end{equation*}
	So, it is sufficient to prove that
	\begin{equation*}
		L_{g^{c}} (X^{c}(u)+Y^{v}(u)) =L_{g}^{c}\circ \kappa_{Q}(X^{c}(u)+Y^{v}(u)).
	\end{equation*}
	Now, since $\kappa_{Q}$ is a vector bundle isomorphism between the vector bundles $\tau_{TQ}$ and $T\tau_{Q}$, it follows that
	\begin{equation*}
		\kappa_{Q}(X^{c}(u)+Y^{v}(u))=(T_{(u,0(q))}(+))(\kappa_{Q}(X^{c}(u)),\kappa_{Q}(Y^{v}(u))),
	\end{equation*}
	and then from \eqref{int:def:can:invol} and the definition of complete lift of a function we deduce
	\begin{equation*}
		L_{g}^{c}\circ \kappa_{Q}(X^{c}(u)+Y^{v}(u))=(T_{(u,0(q))}(+))(TX(u),\widetilde{Y}^{v}(u))(L_{g}).
	\end{equation*}
	where $(+):TQ\times_{Q}TQ\rightarrow TQ$ in the right-hand side of the equality is the addition on the fibers of the vector bundle $\tau_{Q}:TQ\rightarrow Q$. So we have that
	\begin{equation}\label{first:equality}
		L_{g}^{c}\circ \kappa_{Q}(X^{c}(u)+Y^{v}(u))=(TX(u),\widetilde{Y}^{v}(u))(L_{g}\circ (+)).
	\end{equation}
	Next, let $\sigma:(-\varepsilon,\varepsilon)\rightarrow Q$ be a curve on $Q$ such that
	\begin{equation*}
		\sigma(0)=q, \quad \dot{\sigma}(0)=u
	\end{equation*}
	and $Z:(-\varepsilon,\varepsilon)\rightarrow TQ$ a curve over $\sigma$ satisfying
	\begin{equation}\label{def:Z}
		Z(0)=0(q), \quad \dot{Z}(0)=\widetilde{Y}^{v}(u)=(T_{q}0)(u)+Y^{v}(0(q)).
	\end{equation}
	Then, from \eqref{first:equality}, it follows that
	\begin{equation*}
	L_{g}^{c}\circ \kappa_{Q}(X^{c}(u)+Y^{v}(u))=\left. \frac{d}{dt} \right|_{t=0}(L_{g}\circ (+))(X(\sigma(t)),Z(t)),
	\end{equation*}
	hence
	\begin{equation*}
		\begin{split}
			L_{g}^{c}\circ \kappa_{Q}(X^{c}(u)+Y^{v}(u))=\left. \frac{d}{dt} \right|_{t=0} & \left( \frac{1}{2}g(X(\sigma(t)),X(\sigma(t))) \right. \\
			& \left.+g(X(\sigma(t)),Z(t))+ \frac{1}{2}g(Z(t),Z(t)) \right),
		\end{split}
	\end{equation*}
	Thus, using \eqref{def:Z} and the following equalities
	\begin{equation*}
	\begin{split}
	\left. \frac{d}{dt} \right|_{t=0} g(X(\sigma(t)),X(\sigma(t)))& =\dot{\sigma}(0)(g(X,X)), \\
	\left. \frac{d}{dt} \right|_{t=0} g(X(\sigma(t)),Z(t)) & =(\Le_{\dot{\sigma}} g)|_{t=0}(X(q),Z(0)) \\
	 & +g(\Le_{\dot{\sigma}}X(\sigma(t))|_{t=0},Z(0))+g(X(q),\Le_{\dot{\sigma}}Z(t)|_{t=0}), \\
	 \Le_{\dot{\sigma}}Z(t)|_{t=0} & = Y(q),
	\end{split}
	\end{equation*}
	we deduce that
	\begin{equation*}
		L_{g}^{c}\circ \kappa_{Q}(X^{c}(u)+Y^{v}(u))=\frac{1}{2}u(g(X,X))+g(X(q),Y(q)).
	\end{equation*} 
	On the other hand, from \eqref{def:complete:lift:g}, we deduce that
	\begin{equation*}
		\begin{split}
			L_{g^{c}}(X^{c}(u)+Y^{v}(u)) & = \frac{1}{2}g^{c}(u)(X^{c}(u)+Y^{v}(u),X^{c}(u)+Y^{v}(u)) \\
				& = \frac{1}{2}(g(X,X))^{c}(u)+g(X(q),Y(q)) \\
				& = \frac{1}{2}u(g(X,X))+g(X(q),Y(q)),
		\end{split}
	\end{equation*}
	which proves the Lemma.
\end{proof}

In local coordinates, the Lagrangian function associated to $g^{c}$ has the local expression
\begin{equation}\label{complete:g:Lagrangian}
L_{g^{c}}(q,\dot{q},v,\dot{v})=\frac{1}{2}\dot{q}^{k}\frac{\partial g_{ij}}{\partial q^{k}} v^{i}v^{j}+ g_{ij}v^{i} \dot{v}^{j}.
\end{equation}

\begin{proposition}\label{liftedforms}
	Let $g$ be a Riemannian metric on $Q$, $g^{c}$ the complete lift of $g$ to $TQ$, $L_{g}:TQ\rightarrow\R$ and $L_{g^{c}}:TTQ\rightarrow \R$ the corresponding Lagrangian functions. Then
	\begin{equation*}
		\theta_{L_{g^{c}}}=\kappa_{Q}^{*}(\theta_{L_{g}}^{c}), \ \omega_{L_{g^{c}}}=\kappa_{Q}^{*}(\omega_{L_{g}}^{c}), \text{ and } E_{L_{g^{c}}}=(E_{L_{g}})^{c}\circ\kappa_{Q},
	\end{equation*}
	where $\theta_{L_{g}}$ (respectively $\theta_{L_{g^{c}}}$), $\omega_{L_{g}}$ (respectively $\omega_{L_{g^{c}}}$) and $E_{L_{g}}$ (respectively $E_{L_{g^{c}}}$) are the Poincaré-Cartan 1-form, the Poincaré-Cartan 2-form  and the Lagrangian energy associated with $L_{g}$ (respectively with $L_{g^{c}}$).
\end{proposition}

\begin{proof}
	Notice that once we establish that $\theta_{L_{g^{c}}}=\kappa_{Q}^{*}(\theta_{L_{g}}^{c})$, then the corresponding formula for the Poincaré-Cartan 2-forms follows since the pullback commutes with the differential and the differential of the complete lift behaves according to \eqref{differential:completelift}.
	
	It is sufficient to prove $\theta_{L_{g^{c}}}=\kappa_{Q}^{*}(\theta_{L_{g}}^{c})$ over the vector fields $(X^{c})^{c}$, $(X^{c})^{v}$, $(X^{v})^{c}$ and $(X^{v})^{v}$, with $X\in \mathfrak{X}(Q)$. Now,
	\begin{equation*}
		\theta_{L_{g^{c}}}((X^{c})^{c})=dL_{g^{c}}((X^{c})^{v})=d(L_{g}^{c}\circ \kappa_{Q})((X^{c})^{v})=(\kappa_{Q})^{*}dL_{g}^{c}((X^{c})^{v}),
	\end{equation*}
	where we used the definition of the Poincaré-Cartan 1-form and the properties in \eqref{Slifts} in the first equality, while we used Lemma \ref{liftedLagrangianlemma} in the second equality. Then using the definition of pullback together with \eqref{diff:can:invol} we get
	\begin{equation*}
		\theta_{L_{g^{c}}}((X^{c})^{c})=(dL_{g}^{c}\circ \kappa_{Q})((X^{v})^{c}\circ \kappa_{Q})=(dL_{g}^{c}((X^{v})^{c}))\circ \kappa_{Q}=(dL_{g}(X^{v}))^{c}\circ \kappa_{Q},
	\end{equation*}
	where we used \eqref{1formcomplete} and \eqref{differential:completelift} in the last equality. Now, using the definition of $\theta_{L_{g}}$ followed by \eqref{1formcomplete} again we obtain
	\begin{equation*}
	\theta_{L_{g^{c}}}((X^{c})^{c})=(\theta_{L_{g}}(X^{c}))^{c}\circ \kappa_{Q}=\theta_{L_{g}}^{c}((X^{c})^{c})\circ \kappa_{Q}.
	\end{equation*}
	Finally, using \eqref{diff:can:invol} we obtain
	\begin{equation*}
	\theta_{L_{g^{c}}}((X^{c})^{c})=(\kappa_{Q}^{*}(\theta_{L_{g}}^{c}))((X^{c})^{c}).
	\end{equation*}
	By applying the same arguments we may deduce the remaining expressions. Indeed
	\begin{equation*}
		\theta_{L_{g^{c}}}((X^{c})^{v})=0=(\theta_{L_{g}}(X^{v}))^{c}\circ \kappa_{Q}=\theta_{L_{g}}^{c}((X^{v})^{c})\circ \kappa_{Q}=(\kappa_{Q}^{*}(\theta_{L_{g}}^{c}))((X^{c})^{v}),
	\end{equation*}
	where we used the definition of the Poincaré-Cartan 1-forms and \eqref{1formcomplete}, \eqref{Slifts} and \eqref{diff:can:invol}. Also using these arguments and Lemma \ref{liftedLagrangianlemma} we deduce
	\begin{equation*}
		\begin{split}
			\theta_{L_{g^{c}}}((X^{v})^{c}) & =dL_{g^{c}}((X^{v})^{v})=d(L_{g}^{c}\circ \kappa_{Q})((X^{v})^{v})=(\kappa_{Q})^{*}dL_{g}^{c}((X^{v})^{v}) \\
			& =(dL_{g}^{c}((X^{v})^{v}))\circ \kappa_{Q}=(dL_{g}(X^{v}))^{v}\circ \kappa_{Q},
		\end{split}
	\end{equation*}
	where the last equality follows from \eqref{alphalift}. Then
	\begin{equation*}
	\theta_{L_{g^{c}}}((X^{v})^{c})=(\theta_{L_{g}}(X^{c}))^{v}\circ \kappa_{Q}=\theta_{L_{g}}^{c}((X^{c})^{v})\circ \kappa_{Q}=(\kappa_{Q}^{*}(\theta_{L_{g}}^{c}))((X^{v})^{c}).
	\end{equation*}
	The last expression follows directly from \eqref{alphalift} and \eqref{Slifts}.
	
	On the other hand,
	\begin{equation*}
	\theta_{L_{g^{c}}}((X^{v})^{v})=0=(\theta_{L_{g}}(X^{v}))^{v}\circ \kappa_{Q}=\theta_{L_{g}}^{c}((X^{v})^{v})\circ \kappa_{Q}=(\kappa_{Q}^{*}(\theta_{L_{g}}^{c}))((X^{v})^{v}).
	\end{equation*}
	
	To prove the expression relating the Lagrangian energies, denote first by $\Delta_{TQ}\in\mathfrak{X}(TQ)$ the Liouville vector field on the tangent bundle and by $\Delta_{TTQ}\in\mathfrak{X}(TTQ)$ the Liouville vector field on the double tangent bundle. Recall the definition of Lagrangian energy
	\begin{equation*}
		E_{L_{g^{c}}}=\Delta_{TTQ}(L_{g^{c}})-L_{g^{c}} \text{ and } E_{L_{g}}=\Delta_{TQ}(L_{g})-L_{g}.
	\end{equation*}
	Moreover, note that since the Lagrangian functions are of kinetic type, then
	\begin{equation*}
		\Delta_{TTQ}(L_{g^{c}})=2L_{g}^{c} \text{ and } \Delta_{TQ}(L_{g})=2L_{g}.
	\end{equation*}
	Then, using Lemma \ref{liftedLagrangianlemma} we have that
	\begin{equation*}
	E_{L_{g^{c}}}=L_{g^{c}}=L_{g}^{c}\circ \kappa_{Q}=(E_{L_{g}})^{c}\circ \kappa_{Q}.
	\end{equation*}
\end{proof}
{Finally, we will describe the dynamics associated with the Lagrangian function $L_{g^c}: TTQ \to \R$ in terms of the complete lift of the geodesic flow associated with $g$. In addition, we will prove that the trajectories of the Lagrangian system $(TTQ, L_{g^c})$ are just the Jacobi fields of the Riemannian metric $g$.}
\begin{proposition}
		Let $g$ be a Riemannian metric on $Q$, $g^{c}$ the complete lift of $g$ to $TQ$, $L_{g}:TQ\rightarrow\R$ and $L_{g^{c}}:TTQ\rightarrow \R$ the corresponding Lagrangian functions, $\Gamma_{L_{g}}$ and $\Gamma_{L_{g^{c}}}$ the corresponding Lagrangian vector fields on $TQ$ and $TTQ$, respectively. Then
		\begin{enumerate}
			\item $\Gamma_{L_{g^{c}}}=T\kappa_{Q}\circ \Gamma_{L_{g}}^{c} \circ \kappa_{Q}$;
			\item If $Z:I\rightarrow TQ$ is a trajectory of the SODE $\Gamma_{L_{g^{c}}}$, then $Z$ is a Jacobi field for $g$ over a geodesic $c_{v}:I\rightarrow Q$ of $g$.
		\end{enumerate}
\end{proposition}

\begin{proof}
	Let $w\in TTQ$ and $X\in T_{w}(TTQ)$. Then we have that
	\begin{equation*}
		\left(i_{(T_{w}\kappa_{Q})(\Gamma_{L_{g}}^{c}(w))} \omega_{L_{g^{c}}}(\kappa_{Q}(w))\right)(T_{w}\kappa_{Q}(X))=\left(\kappa_{Q}^{*}\omega_{L_{g^{c}}}(w)\right) (\Gamma_{L_{g}}^{c}(w),X).
	\end{equation*}
	Using Proposition \ref{liftedforms}, the last expression reduces to
	\begin{equation*}
		\left( \omega_{L_{g}}^{c}(w) \right)(\Gamma_{L_{g}}^{c}(w),X)=\left( i_{\Gamma_{L_{g}}^{c}(w)}\omega_{L_{g}}^{c}(w) \right)(X)=\left( i_{\Gamma_{L_{g}}}\omega_{L_{g}} \right)^{c}(w)(X),
	\end{equation*}
	where we used \eqref{inner:complete} in the last equality. Using now the geometric equation of motion the last line becomes
	\begin{equation*}
		\left( dE_{L_{g}} \right)^{c}(w)(X)=\left( dE_{L_{g}}^{c} \right)(w)(X)=\left( d(E_{L_{g^{c}}}\circ \kappa_{Q} ) \right)(w)(X),
	\end{equation*}
	where we used \eqref{differential:completelift} and Proposition \ref{liftedforms}. Using now the definition of pullback and the geometric equations of motion we get
	\begin{equation*}
		\left( dE_{L_{g^{c}}} (\kappa_{Q}(w))\right)(T_{w}\kappa_{Q}(X))=\left( i_{\Gamma_{L_{g^{c}}}\kappa_{Q}(w)}\omega_{L_{g^{c}}} (\kappa_{Q}(w))\right)(T_{w}\kappa_{Q}(X)).
	\end{equation*}
	Therefore, since $\omega_{L_{g^{c}}}$ is non-degenerate we deduce
	\begin{equation*}
		\Gamma_{L_{g^{c}}}(\kappa_{Q}(w))=(T_{w}\kappa_{Q})(\Gamma_{L_{g}}^{c})(w),
	\end{equation*}
	from where the first item follows.
	
	Next we prove the second item. By Proposition \eqref{liftedforms}, if $Z:I\rightarrow TQ$ is a curve on the tangent bundle such that $\kappa_{Q}\circ \dot{Z}:I\rightarrow TTQ$ is an integral curve of $\Gamma_{L_{g}}^{c}$, then
	\begin{equation*}
		\kappa_{Q}\circ \dot{Z}(t)=\left( T_{Z(0)}\phi_{t}^{\Gamma_{L_{g}}} \right) (	\kappa_{Q}\circ \dot{Z}(0)),
	\end{equation*}
	where $\phi_{t}^{\Gamma_{L_{g}}}$ is the flow of the vector field $\Gamma_{L_{g}}$. Projecting the previous equation using $T\tau_{Q}$ we get that
	\begin{equation*}
		\tau_{TQ}\circ \dot{Z}(t)=\left( T_{Z(0)}(\tau_{Q} \circ \phi_{t}^{\Gamma_{L_{g}}}) \right) (\kappa_{Q}\circ \dot{Z}(0)),
	\end{equation*}
	One one hand, note that $\tau_{TQ}\circ \dot{Z}(t)$ is just the curve $Z(t)$, by definition of tangent vector field to a curve. On the other hand, let $V:I\rightarrow TQ$ be a curve with initial velocity such that $V'(0)=\kappa_{Q}\circ \dot{Z}(0)$. Then, the expression above may be rewritten as
	\begin{equation*}
		Z(t)=\left. \frac{d}{ds} \right|_{s=0}\left( (\tau_{Q} \circ \phi_{t}^{\Gamma_{L_{g}}}) (V(s))\right).
	\end{equation*}
	Therefore, $Z$ is an infinitesimal variation vector field for a family of trajectories of $\Gamma_{L_{g}}$. Hence, it is a Jacobi field for the Riemannian metric $g$.
\end{proof}

\section{Nonholonomic Jacobi fields for mechanical Lagrangian systems}\label{appd}

Suppose we are given a nonholonomic system $(L,\D)$  with Lagrangian function of the form
\begin{equation}\label{mechanical:Lagrangian}
L(v_{q})=L_{h}(v_{q})-V\circ\tau_{Q} (v_{q}),
\end{equation}
where $h$ is a pseudo-Riemannian metric on $Q$ and $V:Q\rightarrow \R$ is called the \textit{potential} energy {(a real $C^{\infty}$-function on $TQ$). $\tau_{Q}:TQ\rightarrow Q$} is the canonical projection.

Given a pseudo-Riemannian metric $h$, define the \textit{gradient vector field with respect to $h$} $\text{grad}_{h} V$
\begin{equation}\label{grad:definition}
	\text{grad}_{h}V=\sharp_{h}(dV),
\end{equation}
where $\sharp_{h}:\Omega^{1}(Q)\rightarrow \mathfrak{X}(Q)$ is the inverse isomorphism of $\flat_{h}$.

First we prove a Lemma that will allow us to extend some results we have already proved for kinetic type Lagrangian functions to the mechanical type Lagrangian functions with little effort.

\begin{lemma}\label{L:lemma}
	Let $h$ be a pseudo-Riemannian metric on $Q$, $V$ a potential function on $Q$ and $L$ a mechanical Lagrangian associated to $h$ and $V$ defined as in \eqref{mechanical:Lagrangian}. If $\theta_{L}$, $\omega_{L}$ and $E_{L}$ are the Poincaré-Cartan 1-form, the Poincaré-Cartan 2-form and the Lagrangian energy with respect to $L$, respectively, then we have that
	\begin{equation*}
		\theta_{L}=\theta_{L_{h}}, \quad \omega_{L}=\omega_{L_{h}}, \quad E_{L}=L_{h}+V\circ \tau_{Q}.
	\end{equation*}
\end{lemma}

\begin{proof}
	By the definition of the Poincaré-Cartan 1-form we have that
	\begin{equation*}
		\theta_{L}=S^{*}(dL)=S^{*}(dL_{h}-dV^{v}).
	\end{equation*}
	Note that we used that $V^{v}=V\circ \tau_{Q}$. By \eqref{1formvertical} and since the pullback commutes with the differential one has that
	\begin{equation*}
		(dV)^{v}=d(V^{v}),
	\end{equation*}
	for every $V\in C^{\infty}(Q)$. Then, applying \eqref{S*lifts} we deduce
	\begin{equation*}
		\theta_{L}=S^{*}(dL)=S^{*}(dL_{h})=\theta_{L_{h}}.
	\end{equation*}
	Moreover the equality for the Poincaré-Cartan 2-form follows directly from the above.

	As for the Lagrangian energy just observe that
	\begin{equation*}
		\Delta(V^{v})=0,\quad \text{and} \quad \Delta(L_{h})=2L_{h},
	\end{equation*}
	where $\Delta$ is the Liouville vector field on $TQ$. Then it is clear that
	\begin{equation*}
		E_{L}=\Delta(L)-L=L_{h}+V\circ \tau_{Q}.
	\end{equation*}
\end{proof}

Now we will prove a result which is analogous to Theorem \ref{regularityth}.

\begin{theorem}\label{L:regular}
	Given $L$ a mechanical Lagrangian on $Q$ associated to a pseudo-Riemannian metric $h$ and a potential function $V$ defined as in \eqref{mechanical:Lagrangian} and a distribution $\D$, the nonholonomic system $(L,\D)$ is regular if and only if the distribution $\D$ is non-degenerate, in the sense of Theorem \ref{regularityth}.
\end{theorem}

\begin{proof}
	Note that Theorem \ref{regularityth} is a consequence of the Lemma preceding it, which in turn is a consequence of the first relation in \eqref{omegalift}, which by Lemma \ref{L:lemma} remains unchanged for mechanical type Lagrangian functions. Therefore, the proof of Theorem \ref{regularityth} extends to this case.
\end{proof}

Now we prove the result analogous to Theorem \ref{nhgeodesicth}. We characterize the trajectories of $\Gamma_{(L,\D)}$ as the solution of an equation involving the nonholonomic connection $\nabla^{nh}$.

\begin{theorem}\label{Newton:2nd:law}
	Let $L$ be a mechanical Lagrangian on $Q$ associated to a pseudo-Riemannian metric $h$ and a potential function $V$ defined as in \eqref{mechanical:Lagrangian}. If $\D$ is a non-degenerate distribution then the trajectories of $\Gamma_{(L,\D)}$ are solutions of
	\begin{equation}\label{eq:2nd:law}
		\nabla^{nh}_{\dot{c}_{v}} \dot{c}_{v}=-P(\emph{\text{grad}}_{h} V \circ c_{v}), \quad \dot{c}_{v}\in\D
	\end{equation}
	with initial conditions in $\D$, where $\nabla^{nh}$ is the nonholonomic connection associated to $h$ and $P:TQ\rightarrow \D$ is the orthogonal projection onto $\D$.
\end{theorem}

\begin{remark}
	We remark that, in the absence of constraints (where $\D=TQ$), then trajectories of $\Gamma_{(L,TQ)}$ are just the trajectories of $\Gamma_{L}$, the Lagrangian vector field associated to $L$. And a curve $c_{v}:I\rightarrow Q$ is a trajectory of $\Gamma_{L}$ if and only if it satisfies
	\begin{equation*}
	\nabla^{h}_{\dot{c}_{v}} \dot{c}_{v}=\text{grad}_{h} V \circ c_{v},
	\end{equation*}
	where $h$ is the pseudo-Riemannian metric to which the kinetic part of $L$ is associated and $\nabla^{h}$ is the corresponding Levi-Civita connection.
\end{remark}

\begin{proof}
	We may follow the proof of Theorem \ref{nhgeodesicth} making the appropriate changes. The first change is to substitute $L_{h}$ by $E_{L}$ on the geometric equation. Then, following the same arguments we will eventually get
	\begin{equation*}
		\Gamma_{(L_{h},\D)}(u)(\widehat{\flat_{h}(X)})=h(\nabla^{h}_u X, u)-X^{c}(u)(V\circ\tau_{Q}),
	\end{equation*}
	which by \eqref{completelift} is equivalent to
	\begin{equation*}
		\Gamma_{(L,\D)}(u)(\widehat{\flat_{h}(X)})=h(\nabla^{h}_u X, u)-dV(X)\circ\tau_{Q}(u).
	\end{equation*}

	Then, evaluating the last equation over the curve $\dot{c}_{v}$, noting that $\Gamma_{(L_{h},\D)}(\dot{c}_{v})$ is just $\ddot{c}_{v}$ and using \eqref{grad:definition}, we deduce
	\begin{equation*}
		\ddot{c}_{v}\left( \widehat{\flat_{h}(X)} \right)=	h(\nabla^{h}_{\dot{c}_{v}}X,\dot{c}_{v})-h(\text{grad}_{h} V\circ c_{v},X\circ c_{v}).
	\end{equation*}
	As in the proof of Theorem \ref{nhgeodesicth}, the last expression reduces to
	\begin{equation*}
		h(X\circ c_{v}, \nabla^{h}_{\dot{c}_{v}}\dot{c}_{v}+\text{grad}_{h} V\circ c_{v})= 0
	\end{equation*}
	and since $X$ is an arbitrary section of $\D$ we deduce that
		\begin{equation*}
	P(\nabla^{h}_{\dot{c}_{v}}\dot{c}_{v}+\text{grad}_{h} V\circ c_{v})= 0,
	\end{equation*}
	which finishes the proof.
\end{proof}
{Now, we will introduce the complete lift of a mechanical nonholonomic system following the same ideas that in the case of a kinetic nonholonomic system.}
\begin{definition}
	If $g$ is a Riemannian metric on $Q$ and $V$ is a potential function forming a mechanical Lagrangian $L$ as in \eqref{mechanical:Lagrangian}, the nonholonomic system $(\tilde{L},\D^{c})$ is the complete lift of the nonholonomic system $(L,\D)$, with the Lagrangian function of mechanical type $\tilde{L}:TTQ\rightarrow \R$ defined by
	\begin{equation*}
		\tilde{L}=L_{g^{c}}-V^{c}\circ \tau_{TQ},
	\end{equation*}
	where $\tau_{TQ}:T(TQ)\rightarrow TQ$ is the canonical projection and $\D^{c}$ is the complete lift of the distribution $\D$.
\end{definition}
{Next, we will prove some results which will be used later}.
\begin{lemma}
	If $V$ is a function on $Q$, its complete lift satisfies
	\begin{equation}
		(V\circ \tau_{Q})^{c}\circ\kappa_{Q}=V^{c}\circ \tau_{TQ}.
	\end{equation}
\end{lemma}

\begin{proof}
	Observe that for any $Y\in TTQ$ we have that
	\begin{equation*}
		(V\circ \tau_{Q})^{c}\circ\kappa_{Q}(Y)=d(V\circ\tau_{Q}) (\kappa_{Q}(Y))
	\end{equation*}
	by \eqref{function:lifts}. Then we have that
	\begin{equation*}
		(V\circ \tau_{Q})^{c}\circ\kappa_{Q}(Y)=dV (T\tau_{Q}\circ\kappa_{Q}(Y)))=dV(\tau_{TQ}(Y))=V^{c}\circ \tau_{TQ} (Y),
	\end{equation*}
	where we used the fact that $\kappa_{Q}$ is a morphism between the vector bundles $T\tau_{Q}$ and $\tau_{TQ}$ and \eqref{function:lifts}.
\end{proof}

\begin{lemma}\label{complete:gradient}
	Let $g$ be a Riemannian metric, $g^{c}$ its complete lift and $V$ a function on $Q$. We have that
	\begin{equation}
		\emph{\text{grad}}_{g^{c}} V^{c}=(\emph{\text{grad}}_{g}V)^{c}.
	\end{equation}
\end{lemma}

\begin{proof}
	On one hand, for an arbitrary $Z\in\mathfrak{X}(Q)$ we have that
	\begin{equation*}
		g^{c}(\text{grad}_{g^{c}} V^{c},Z^{c}) = d(V^{c}) (Z^{c})=(dV(Z))^{c},
	\end{equation*}
	where we used the definition of $\text{grad}_{g^{c}} V^{c}$, \eqref{differential:completelift} and \eqref{1formcomplete}. Then, using the definition of $\text{grad}_{g}V$
	\begin{equation*}
	g^{c}(\text{grad}_{g^{c}} V^{c},Z^{c}) =(g(\text{grad}_{g}V,Z))^{c}=g^{c}((\text{grad}_{g}V)^{c},Z^{c}),
	\end{equation*}
	where we used \eqref{def:complete:lift:g}. On the other hand, using the same arguments we deduce
	\begin{equation*}
	g^{c}(\text{grad}_{g^{c}} V^{c},Z^{v}) = d(V^{c}) (Z^{v})=(dV(Z))^{v}=(g(\text{grad}_{g}V,Z))^{v}.
	\end{equation*}
	Using again \eqref{def:complete:lift:g} we conclude that
	\begin{equation*}
	g^{c}(\text{grad}_{g^{c}} V^{c},Z^{v}) =g^{c}((\text{grad}_{g}V)^{c},Z^{v}).
	\end{equation*}
	Since $g^{c}$ is non-degenerate and $Z^{v}$, $Z^{c}$ are a basis of sections for $\mathfrak{X}(TQ)$ we have finished the proof.
\end{proof}
{Now, we will prove a similar result to Theorem \ref{MainTheorem} for the more general case of a mechanical nonholonomic system.}
\begin{theorem}\label{MainTheoremD}
	Let $(L,\D)$ be a nonholonomic system of mechanical type and $\Gamma_{(L,\D)}$ the associated nonholonomic SODE. Then
	\begin{enumerate}
		\item[(i)] The complete lift $(\tilde{L},\D^{c})$ is a regular nonholonomic system.
		\item[(ii)] Let $\Gamma_{(\tilde{L},\D^{c})}\in \mathfrak{X}(\D^{c})$ be the nonholonomic SODE associated with the system $(\tilde{L},\D^{c})$ and $\kappa_{Q}:TTQ\rightarrow TTQ$ the canonical involution. Then
		\begin{equation}
		\Gamma_{(\tilde{L},\D^{c})}=T\kappa_{Q}|_{T\D} \circ \Gamma_{(L,\D)}^{c} \circ \kappa_{Q}|_{\D^{c}}
		\end{equation}
		and so we have
		\begin{enumerate}
			\item[(a)] $\Gamma_{(\tilde{L},\D^{c})}$ is $T\tau_{Q}|_{\D^{c}}$-projectable over $\Gamma_{(L,\D)}$; \\
			\item[(b)] The trajectories of $\Gamma_{(\tilde{L},\D^{c})}$ are just the Jacobi fields for the nonholonomic system $(L,\D)$.
		\end{enumerate}	
	\end{enumerate}
\end{theorem}

\begin{proof}
	Item (i) is a consequence of Theorem \ref{L:regular} together with the proof of Proposition \ref{regularityprop}, where we see that the distribution $\D^{c}$ is non-degenerate.
	
	Before proving item (ii), note that by Lemma \ref{L:lemma} and Proposition \ref{liftedforms} we have that
	\begin{equation*}
		\omega_{\tilde{L}}=\omega_{L_{g^{c}}}=\kappa_{Q}^{*}\omega_{L_{g}}^{c}=\kappa_{Q}^{*}\omega_{L}^{c}
	\end{equation*}
	and
	\begin{equation*}
		E_{\tilde{L}}=E_{L_{g^{c}}}+V^{c}\circ \tau_{TQ}=(E_{L_{g}})^{c}\circ\kappa_{Q}+(V\circ \tau_{Q})^{c}\circ\kappa_{Q}=E_{L}^{c}\circ\kappa_{Q}.
	\end{equation*}
	
	Now we may follow the proof of Theorem \ref{MainTheorem} and we conclude by following exactly the same steps that
	$(\kappa_{Q})_{*}\Gamma_{(L,\D)}^{c}$ is a vector field on $\D^{c}$ and by uniqueness of nonholonomic vector field, it coincides with $\Gamma_{(\tilde{L},\D^{c})}$, i.e.,
	\begin{equation}
	\Gamma_{(\tilde{L},\D^{c})}=T\kappa_{Q}|_{T\D} \circ \Gamma_{(L,\D)}^{c} \circ \kappa_{Q}|_{\D^{c}}.
	\end{equation}
	
	The remaining statements in item $(ii)$ are just consequences of the properties of the complete lift and the canonical involution and we can follow the proof {of Theorem \ref{MainTheorem}} with the necessary changes:
	\begin{equation*}
	\begin{split}
	TT\tau_{Q}(\Gamma_{(\tilde{L},\D^{c})}) & =T(T\tau_{Q}\circ\kappa_{Q}|_{T\D}) \circ \Gamma_{(L,\D)}^{c} \circ \kappa_{Q}|_{\D^{c}} \\
	& = T(\tau_{TQ}|_{T\D})(\Gamma_{(L,\D)}^{c} \circ \kappa_{Q}|_{\D^{c}}) \\
	& = \Gamma_{(L,\D)}\circ \tau_{TQ}|_{T\D} \circ\kappa_{Q}|_{\D^{c}} \\
	& = \Gamma_{(L,\D)}.
	\end{split}
	\end{equation*}
		
	Now if $W:I\rightarrow TQ$ is a trajectory of $\Gamma_{(\tilde{L},\D^{c})}$, then $\kappa_{Q}\circ\dot{W}:I\rightarrow T\D$ is an integral curve of $\Gamma_{(L,\D)}^{c}$. Therefore we may write it as
	\begin{equation*}
	\kappa_{Q}\circ\dot{W}(t)=\left( T_{W(0)} \phi_{t}^{\Gamma_{(L,\D)}} \right) (\kappa_{Q}\circ\dot{W}(0)).
	\end{equation*}
	So,
	\begin{equation*}
	W(t)=T\tau_{Q}(\kappa_{Q}\circ\dot{W})=T\tau_{Q}\left( \left( T_{W(0)} \phi_{t}^{\Gamma_{(L,\D)}} \right) (\kappa_{Q}\circ\dot{W}(0)) \right)
	\end{equation*}
	and
	\begin{equation*}
	W(t)= \left( T_{W(0)} (\tau_{Q}\circ \phi_{t}^{\Gamma_{(L,\D)}}) \right) (\kappa_{Q}\circ\dot{W}(0)).
	\end{equation*}
	Let now $v:I\rightarrow \D$ be a curve such that its initial velocity is $v'(0)=\kappa_{Q}\circ\dot{W}(0)$. Then
	\begin{equation*}
	W(t)= \left. \frac{d}{ds} \right|_{s=0} \left( \tau_{Q}\circ \phi_{t}^{\Gamma_{(L,\D)}}\right) (v(s)).
	\end{equation*}
	Hence, $W$ is a nonholonomic Jacobi field for $\Gamma_{(L,\D)}$, since it is an infinitesimal variation of nonholonomic trajectories of $\Gamma_{(L,\D)}$.
	
\end{proof}
{Finally, we will present the Jacobi equation for the nonholonomic Jacobi fields associated with a mechanical nonholonomic system.}
\begin{theorem}
	Let $(L,\D)$ be a mechanical nonholonomic system, $\nabla^{nh}$ the nonholonomic connection on $Q$ with torsion and curvature tensors denoted by $T^{nh}$ and $R^{nh}$, respectively, and $W:I\rightarrow TQ$ a vector field along a nonholonomic trajectory $c:I\rightarrow Q$. Then $W$ is a nonholonomic Jacobi field if and only if 
	\begin{equation}\label{nhjacobiV}
		\begin{split}
			& \nabla^{nh}_{\dot{c}}\nabla^{nh}_{\dot{c}}W+\nabla^{nh}_{\dot{c}}T^{nh}(W,\dot{c})+R^{nh}(W,\dot{c})\dot{c}+\nabla^{nh}_{W}(P(\emph{\text{grad}}_{g}V\circ c))=0, \\
			& \dot{W}(t)\in \D^{c}.
		\end{split}
	\end{equation}
\end{theorem}

\begin{proof}
	We already know by Theorem \ref{Newton:2nd:law} that if $c_{v}:I\rightarrow Q$ is a trajectory of $\Gamma_{(L, \D)}$, then it satisfies equations \eqref{eq:2nd:law}. Moreover, by Theorem \ref{MainTheoremD} if $W:I\rightarrow TQ$ is a Jacobi field for the nonholonomic system $(L,\D)$, it is a trajectory of the mechanical Lagrangian $\Gamma_{(\tilde{L},\D^{c})}$. As a result, $W$ must satisfy the equations
	\begin{equation*}
		\nabla^{NH}_{\dot{W}} \dot{W}=-P^{T}(\text{grad}_{g^{c}} V^{c}\circ W), \quad \dot{W}\in \D^{c},
	\end{equation*}
	where $\nabla^{NH}$ is the linear connection on $TQ$ defined by
	\begin{equation*}
	\nabla^{NH}_{X} Y:=P^{T}(\nabla_{X}^{g^{c}} Y)+\nabla^{g^{c}}_{X}[P'^{T}(Y)], \quad \text{for} \ X,Y \in \mathfrak{X}(TQ),
	\end{equation*}
	where $\nabla^{g^{c}}$ is the Levi-Civita connection of $g^{c}$, $P^{T}:TTQ\rightarrow \D^{c}$ is the associated  orthogonal projector onto the distribution $\D^{c}$ and $P'^{T}:TTQ\rightarrow (\D^{c})^{\bot}$ is the orthogonal projector onto $(\D^{c})^{\bot}$, the orthogonal distribution.
	
	On one hand, by Proposition \ref{completeconnection}, $\nabla^{NH}=\nabla^{c}$. On the other hand, by Lemma \ref{complete:gradient}, we have that
	\begin{equation*}
		\text{grad}_{g^{c}} V^{c}=(\text{grad}_{g}V)^{c}
	\end{equation*}
	and by Lemma \ref{Projectorlift} we have that
	\begin{equation*}
		P^{T}(X^{c})=(P(X))^{c}.
	\end{equation*} 
	Hence, $W$ must satisfy
	\begin{equation*}
		\nabla^{c}_{\dot{W}} \dot{W}=-(P(\text{grad}_{g}V))^{c}\circ W, \quad \dot{W}\in \D^{c}.
	\end{equation*}
	
	Now we will follow Proposition \ref{complete:connection:coordinates} and  keep the same notation that was introduced in the corresponding proof. Suppose the local expression of $P(\text{grad}_{g}V)$ is
	\begin{equation*}
		P(\text{grad}_{g}V)(q^{i})=(P(\text{grad}_{g}V))^{i} \frac{\partial}{\partial q^{i}}.
	\end{equation*}
	
	 Equation \eqref{liftedgeo2} together with
	\begin{equation*}
		(P(\text{grad}_{g}V))^{c}\circ W=(P(\text{grad}_{g}V))^{i}\left. \frac{\partial}{\partial q^{i}}\right|_{W(t)}+W^{j}(t)\frac{\partial (P(\text{grad}_{g}V))^{i}}{\partial q^{j}}\left. \frac{\partial}{\partial \dot{q}^{i}}\right|_{W(t)}
	\end{equation*}
	and the fact that $c_{v}$ satisfies the equations \eqref{eq:2nd:law} imply that
	\begin{equation*}
		\begin{split}
			(\nabla^{nh})^{c}_{\dot{W}}\dot{W}+ & (P(\text{grad}_{g}V))^{c}(W(t)) =\left(\ddot{W}^{k}+\dot{q}^{i}\dot{q}^{j}W^{l}\frac{\partial \Gamma_{ij}^{k}}{\partial q^{l}} \right.\\
			& \left. +\dot{q}^{j}\dot{W}^{i}(\Gamma_{ij}^{k}+\Gamma_{ji}^{k})+W^{j}(t)\frac{\partial (P(\text{grad}_{g}V))^{i}}{\partial q^{j}}\right)\frac{\partial}{\partial \dot{q}^{k}}.
		\end{split}
	\end{equation*}
	Using similar techniques to those applied in the proof of Theorem \ref{Jacobi:equation}, we are able to prove that
	\begin{equation*}
		\begin{split}
			(\nabla^{nh} & )^{c}_{\dot{W}}\dot{W} + (P(\text{grad}_{g}V))^{c}(W(t))= \\
			&\left( \nabla^{nh}_{\dot{c}}\nabla^{nh}_{\dot{c}}W+\nabla^{nh}_{\dot{c}}T^{nh}(W,\dot{c})+R^{nh}(W,\dot{c})\dot{c}+\nabla^{nh}_{W}(P(\text{grad}_{g}V\circ c)) \right)^{v}
		\end{split}
	\end{equation*}
	Indeed, by following the proof and having in mind that now the curve $c_{v}$ locally satisfies the equation
	\begin{equation*}
	\ddot{q}^{i}=-\Gamma_{jk}^{i}\dot{q}^{j}\dot{q}^{k}-(P(\text{grad}_{g}V))^{i},
	\end{equation*}
	we deduce that the sum of $\nabla^{nh}_{\dot{c}}\nabla^{nh}_{\dot{c}}W$ and $R^{nh}(W,\dot{c})\dot{c}$ is
	\begin{equation*}
		\begin{split}
			& \left[ \ddot{W}^{m}+2\dot{W}^{j}\dot{q}^{i}\Gamma_{ij}^{m}+W^{i}\dot{q}^{j}\dot{q}^{l}\Gamma_{jl}^{k}T_{ik}^{m}+W^{j}\dot{q}^{i}\dot{q}^{l}T_{ij}^{k}\Gamma_{lk}^{m} \right.\\
			 & \left. +\dot{q}^{i}\dot{q}^{l}W^{j}\left( \frac{\partial T_{ij}^{m}}{\partial q^{l}}+\frac{\partial \Gamma_{il}^{m}}{\partial q^{j}} \right) -(P(\text{grad}_{g}V))^{i}W^{j}\Gamma_{i j}^{m}\right]\frac{\partial}{\partial q^{m}}.
		\end{split}
	\end{equation*}
	Since
	\begin{equation}
		\begin{split}
			\nabla^{nh}_{\dot{c}}T^{nh}(W,\dot{c})=& \left( \dot{W}^{i}\dot{q}^{j}T_{ij}^{m}-W^{i}\Gamma_{lk}^{j}\dot{q}^{l}\dot{q}^{k}T_{ij}^{m}-W^{i}(P(\text{grad}_{g}V))^{j}T_{i j}^{m} \right. \\
			& \left. -W^{i}\dot{q}^{j}\frac{\partial T_{ji}^{m}}{\partial q^{l}}\dot{q}^{l}-W^{i}\dot{q}^{j}T_{ji}^{k}\dot{q}^{l}\Gamma_{lk}^{m} \right)\frac{\partial}{\partial q^{m}}
		\end{split}
	\end{equation}
	and
	\begin{equation*}
		\nabla^{nh}_{W}(P(\text{grad}_{g}V))=\left( W^{i}\frac{\partial (P(\text{grad}_{g}V))^{j}}{\partial q^{i}}+W^{i}(P(\text{grad}_{g}V))^{k}\Gamma_{i k}^{j} \right)\frac{\partial}{\partial q^{j}}
	\end{equation*}
	we easily see that we obtain the expected result and the theorem follows.	
\end{proof}

\section*{Conclusions and future work}

In this paper, we have introduced  for the first time  a rigorous definition of Jacobi fields for nonholonomic systems in pure Riemannian geometric  terms, we have also characterized them and finally  we have given some equivalent versions of the nonholonomic Jacobi equation (see Table 1).

In a future paper, we will continue this program studying  conjugate points, the possible relation with minimizing properties of nonholonomic geodesics where the exponential nonholonomic map (see \citep*{AMM}) will play an important role.

{Another interesting goal, to be covered in an upcoming publication, is to extend the results of this paper on Jacobi fields to the reduction of nonholonomic mechanical systems with symmetries. This kind of systems have been extensively discussed in the literature (see \citep*{koiller, BKMM96, CaLeMAMa, GG2008, CdLMM2009, GrLeMaMa, LMD2010, Ba1, Ba2, BaFe}).}

Finally, we remark that many of the results in this paper may be extended for Jacobi fields in sub-riemannian geometry \citep*{GhIgMaPa}.

\section*{Acknowledgements}

D. Mart{\'\i}n de Diego and A. Simoes are supported  by I-Link Project (Ref: linkA20079), Ministerio de Ciencia e Innovaci\'on ( Spain) under grants  MTM\-2016-76702-P and ``Severo Ochoa Programme for Centres of Excellence'' in R\&D (SEV-2015-0554). A. Anahory Simoes is supported by the FCT research fellowship SFRH/BD/129882/2017. J.C. Marrero  acknowledges the partial support by European Union (Feder) grant PGC2018-098265-B-C32.


\bibliography{thesisreferences}{}


\end{document}